\documentclass[12pt]{article}

\usepackage[margin=1in]{geometry}
\usepackage{amsmath, amsthm, amssymb}
\usepackage{mathtools}
\usepackage{alphabeta}
\usepackage{bbm}
\usepackage{cite}
\usepackage{hyperref}
\usepackage{enumitem}
\usepackage{xcolor}
\usepackage{enumitem}
\usepackage{comment}
\usepackage{setspace}
\newtheorem{assumption}{Assumption}

\newtheorem{theorem}{Theorem}[section]
\newtheorem{lemma}{Lemma}
\newtheorem{proposition}{Proposition}[section]
\newtheorem{corollary}{Corollary}[section]

\theoremstyle{definition}
\newtheorem{definition}{Definition}

\theoremstyle{remark}
\newtheorem{remark}{Remark}

\begin{document}

\begin{center}
    {\Large \bf Global in time solutions to stochastic reaction-diffusion systems with superlinear reactions satisfying a triangular control of mass} \\[0.5cm]

    \textsc{Dionysis Milesis and Michael Salins} \\
    \vspace{0.6cm}
    \today
\end{center}


\begin{abstract}
We study systems of reaction–diffusion equations perturbed by multiplicative noise, where the reaction terms satisfy quasipositivity, a triangular mass-control structure, and polynomial growth. Our results apply to a broad class of reaction–diffusion systems arising, most notably, in chemistry and biology. In the deterministic setting these assumptions are known to guarantee the global existence of solutions. In the stochastic setting, however, reaction-diffusion systems have typically been analyzed under different assumptions on the reactions that preclude many natural models, such as reversible chemical reaction networks modeled by the mass-action law, and the question of global existence and uniqueness under a mass-control structure has remained open. In this work, we show that stochastically perturbing reaction–diffusion systems with triangular control of mass by suitable multiplicative noise leads to solutions that exist for all time.
\end{abstract}

\section{Introduction}
Let $i\in \{1,...,m\}$ and consider the following reaction--diffusion equations perturbed by multiplicative noise
\begin{equation}\label{eq: initial mxr system}
    \begin{cases}
        \partial_t u_i(t,x) - d_i\Delta u_i(t,x) = f_i(u(t,x))+\sum\limits_{k=1}^r \sigma_{ik}(u(t,x))\dot{W}_k(t,x)\text{ in } (0,T)\times D,\\
        u_i = 0 \text{ or } \dfrac{\partial u_i}{\partial n} = 0 \text{ on } (0,T)\times \partial D,\\
        u_i(0,x) = u_{i0}(x),
    \end{cases}
\end{equation}
where $d_i>0$ are the diffusion coefficients and $D \subset \mathbb R^d$ is an open and bounded domain with sufficiently smooth boundary. If $f$ and $\sigma$ are globally Lipschitz continuous functions, it is a classical result (see, for instance, \cite{da2014stochastic} or \cite{dalang2026stochastic}) that \eqref{eq: initial mxr system} has a unique global solution. In \cite{cerrai2003stochastic}, Cerrai proved existence and uniqueness under more general conditions on the reactions. In particular, Cerrai considered reaction terms that could be written as
\[
f_i(a) = g_i(a) + q_i(a_i),
\]
where $g_i(a)$ is a locally Lipschitz continuous function of the whole system with linear growth and $q_i(a_i)$ is a real function of the $i$th component only satisfying a strong dissipative condition of the form
\[
\bigl(q_i(s+r)-q_i(s)\bigr)r \leq -a|r|^{\beta+1} +C(1+|s|^{\beta+1}),
\]
for some $a,\beta>0$ and $C\geq 0$. Some years later, the second author of the present paper relaxed this dissipativity condition of Cerrai by assuming $q_i$ to be a decreasing function \cite{salins2021systemsLDP}. The prototypical examples both Cerrai and the second author of the present paper had in mind were polynomials of odd order with negative leading coefficient of the form 
\[
f_i(a) = -a_i^{2k+1}+p_i(a),
\]
where $p_i(a)$ are polynomials on $\mathbb R^m$ with degrees less than or equal to $2k$.

Without stochastic perturbation, reaction-diffusion equations have  been studied under a substantially different set of assumptions on the reactions encoded in either of the terms \textit{mass control}, \textit{mass dissipation}, or \textit{mass conservation}. By \textit{mass control} we mean that the 
reactions satisfy 
\begin{equation}\label{eq: M}
f_1(a)+...+f_m(a)\leq C(1+a_1+....+a_m),\tag{M}
\end{equation}
where $C\in \mathbb R$. In the context of reaction-diffusion systems with control of mass structures, the solutions $u_i(t,x)$ can be interpreted as molar concentrations of chemicals. Therefore, the quantity
\begin{equation}\label{eq: Total mass of system}
\sum_{i=1}^m \int_D u_i(t,x)dx
\end{equation}
is the total mass of the system, and \eqref{eq: M} implies that \eqref{eq: Total mass of system} is bounded above by 
\[
c(e^{tC}-1) + e^{tC}\int_D \sum_{i=1}^m u_i(0,x)dx.
\]
Therefore, the mass control structure \eqref{eq: M} implies that the total mass of the system remains bounded on any interval (see, for instance, the calculations leading to (1.9) in \cite{pierre2010global}). In particular, if $C=0$, then the total mass of the system is conserved.

The term \textit{mass control} is primarily motivated by systems modeling reversible chemical reactions using the mass-action law. For example, a reversible reaction between  the reacting species $A,B,C,D$ with molar concentrations $a,b,c,d$ that has the form
\[
A+B \xrightleftharpoons{} C+D
\]
may be described \cite{pierre2010global}, using the mass-action law, by the reactions
\begin{equation}\label{eq: A+B-><-C+D reactions}
\begin{cases}
    f_1(a,b,c,d) = -k_1ab+k_2cd,\\
    f_2(a,b,c,d) = -k_1ab+k_2cd,\\
    f_3(a,b,c,d) = k_1ab-k_2cd,\\
    f_4(a,b,c,d) = k_1ab-k_2cd,
\end{cases}
\end{equation}
where $k_1,k_2>0$.
\textit{Mass control} in this case  refers to the property  $\sum\limits_{i=1}^4 f_i(a,b,c,d) = 0,$ which may be thought of as mass conservation. Examples of such chemical reaction networks include equilibrium esterification and ketalization processes; for instance, acetic acid and ethanol reacting reversibly to form ethyl acetate and water,
$CH_3COOH + C_2H_5OH \rightleftharpoons CH_3COOC_2H_5 + H_2O$
\cite{esterification}, and glycerol reacting with acetone to form solketal and water, $
C_3H_8O_3 + C_3H_6O \rightleftharpoons C_6H_{12}O_3 + H_2O$
\cite{solketal}. 

Another instance of a reversible reaction is
\[
A+B \xrightleftharpoons{} C,
\]
which is modeled \cite{A+B-><-C}, using the mass-action law, by the reactions
\begin{equation}\label{eq: A+B-><- C reactions}
\begin{cases}
    f_1(a,b,c) = -m_1ab+m_2c,\\
    f_2(a,b,c)=-m_1ab+m_2c,\\
    f_3(a,b,c)= m_1ab-m_2c,
\end{cases}
\end{equation}
where $m_1,m_2>0$. Besides the total mass control $f_1+f_2+f_3\leq -m_2c$ (recall that $a,b,c$, modeling concentrations, can be assumed nonnegative), the system satisfies a \textit{triangular mass control} as well. This corresponds to the successive inequalities $f_1\leq m_2c$, $f_1+f_2\leq m_2c$, and $f_1+f_2+f_3\leq -m_2c$. For the general definition of the triangular  mass control structure, see \eqref{eq: TMC} below. Examples of chemical and biological reactions of the form $A+B\xrightleftharpoons{} C$ include the hydration of carbon dioxide $CO_2+H_2O \xrightleftharpoons{} H_2CO_3$ \cite{hydrationofcarbondioxide}, and reversible binding processes such as oxygen binding to myoglobin, $Mb+O_2 \xrightleftharpoons{} MbO_2$ \cite{ligandMb}. The \textit{triangular mass control} structure, which is the main object of interest in this paper, appears naturally in more complicated chemical reaction networks. A prominent example is the Michaelis-Menten mechanism for enzyme kinetics \cite{MichaelisMenten1,MichaelisMenten2} which can be symbolically represented as
\[
S+E\xrightleftharpoons{} C \rightarrow P+E
\]
and describes a substrate $S$
that reacts reversibly with an enzyme $E$ to form a complex $C$ which is
transformed irreversibly into a product $P$ and the enzyme $E$. The reactions describing this process are
\begin{equation}\label{eq: S+E-><-C->P+E}
\begin{cases}
f_1(s,e,c,p) &= -k_1 se + k_3 c, \\
f_2(s,e,c,p)&= -k_1se + (k_3 + k_2)c, \\
f_3(s,e,c,p)&= k_1se - (k_3 + k_2)c\\
f_4(s,e,c,p)&= k_2c,
\end{cases}
\end{equation}
where $k_1,k_2,k_3>0$, and $s,e,c,$ and $p$ are the molar concentrations of $S,E,C,$ and $P$ respectively. 

Deterministic reaction-diffusion systems whose reaction terms satisfy \eqref{eq: M}, or some similar form of mass control structure, have received thorough treatment in the past twenty years; for an indicative list, we refer to \cite{masscontrol1,masscontrol2,masscontrol3,A+B-><-C,masscontrol5,masscontrol6,masscontrol7,masscontrol8,masscontrol9,masscontrol10,masscontrol11,masscontrol13,pierre2000blowup,application1,application2,application3,application4,application5,exponentialreact,pierre2010global,bothe2010quasi}. The literature in the case when random perturbation is imposed on the system, however, is, to the best of our knowledge, rather scarce. Up to date, the most basic question regarding the  existence of solutions to the stochastically perturbed mass--controlled system \eqref{eq: initial mxr system} has remained open, with the only result in this direction to have been given recently by Agresti \cite{agresti2024delayed}. In his paper, Agresti showed that adding a transport noise that is the sum of Brownian motions paired with a divergence--free field to mass-controlled reaction diffusion systems satisfying \eqref{eq: M} can have a regularizing effect, in the sense that it can delay the blow-up of strong solutions arbitrarily far in time. Under an
additional strong assumption (that the total mass of the system decays exponentially fast), Agresti proved 
that the solution is global in time with high probability.

In this paper, we answer the question of global-in-time existence for \eqref{eq: initial mxr system} in the case where the noise $\dot{W}$ is space-time white when $d=1$, and white in time and colored in space when $d\geq 2$, and the reaction terms satisfy, among other conditions, the \textit{triangular mass control}, in the sense described for the sets of reactions \eqref{eq: A+B-><- C reactions} and \eqref{eq: S+E-><-C->P+E}; see \eqref{eq: TMC} below. We mention here that reactions of the form given by \eqref{eq: A+B-><-C+D reactions} cannot be treated yet by the methods employed in this paper. To prove global existence in the case of a triangular mass control, we will closely  follow the framework  Pierre established in the survey \cite{pierre2010global} that gathers basic global existence results for deterministic mass-controlled systems. Our way of manipulating the reaction terms of \eqref{eq: initial mxr system} to prove global existence is inspired, in particular, by Pierre's arguments.

 We shall see now how the basic assumptions satisfied by mass-controlled systems in Pierre's framework differ from those of Cerrai \cite{cerrai2003stochastic}. In Theorem 3.5. of \cite{pierre2010global}, Pierre first assumes the reactions to be \textit{quasipositive} in the sense that,
\begin{equation*}
    f_1(0,a_2,...,a_m)\geq 0,\quad f_2(a_1,0,a_3,...,a_m)\geq 0, \quad \cdots \quad f_m(a_1,...,a_{m-1},0)\geq 0.\tag{P}\label{eq: P}
\end{equation*}
In the deterministic setting, \eqref{eq: P} implies that the solutions remain nonnegative whenever the initial data are nonnegative. It is worth noting that Cerrai's assumptions \cite{cerrai2003stochastic} allowed also for negative solutions. Nevertheless, it is typically the case that solutions to reaction--diffusion systems with control of mass model concentrations (for instance, of chemicals), and therefore assuming them nonnegative is rather reasonable.

Next, Pierre assumes that the reactions have a structure of \textit{triangular mass control}, in the sense that
\begin{align*}
\begin{split}
    &f_{1}(a)\leq C_1(1+a_1+...+a_m),\\
    &f_1(a)+f_2(a)\leq C_2(1+a_1+...+a_m),\\
    &\vdots\\
    &f_1(a)+...+f_m(a)\leq C_m(1+a_1+...+a_m).
\end{split}
\tag{TMC}\label{eq: TMC}
\end{align*}
Notice that the last row of \eqref{eq: TMC} is actually \eqref{eq: M}. It is not at all apparent why one would require a stronger structure such as \eqref{eq: TMC} instead of relying on \eqref{eq: M} for obtaining global-in-time existence. In fact, a famous paper from the deterministic literature \cite{pierre2000blowup} demonstrates, by constructing an example, that assumptions \eqref{eq: P} and \eqref{eq: M} are not always enough to prevent the explosion of solutions. What Pierre shows in his survey \cite{pierre2010global} is that triangular mass control \eqref{eq: TMC} serves as one sufficient condition that guarantees, in the deterministic setting, that solutions do not blow up in finite time.

When $a_1,...,a_m$ are nonnegative, Cerrai's assumptions \cite{cerrai2003stochastic} trivially imply \eqref{eq: TMC} as the interaction terms Cerrai considers are assumed to grow linearly. However, Cerrai's assumptions do not allow for nonlinear interaction terms that may cancel each other when added row-wise as can very often happen in \eqref{eq: TMC}. A simple instance of such reaction terms is
\begin{equation*}\label{eq: prototypical example}
\begin{cases}
    f_1(u_1,u_2)= -u_1u_2,\\
    f_2(u_1,u_2)= u_1u_2.
\end{cases}
    \end{equation*}
Other examples that are precluded from Cerrai's assumptions are the one describing the chemical reactions \eqref{eq: A+B-><- C reactions} and \eqref{eq: S+E-><-C->P+E}. More examples of reactions satisfying a triangular mass control structure with highly nonlinear, coupled interactions are presented in Section \ref{section: assumptions}. 

Finally, along with \eqref{eq: P} and \eqref{eq: TMC}, Pierre also imposes a restriction on the growth of the reactions, namely that they grow polynomially, an assumption that we adopt as well. This condition was also used by Cerrai \cite{cerrai2003stochastic}. It is expected that this condition can be relaxed significantly. We refer, for instance, to the paper \cite{exponentialreact} in which global existence is established for the deterministic system described by the reactions $f_1(u_1,u_2) = -u_1(1+u_2)e^{u_2^2}$ and $ f_2(u_1,u_2) = u_1e^{u_2^2}$ which grow faster than exponentially. The complete list of assumptions we impose on $f_i$, $\sigma_{ik}$ and the noise $\dot{W}$ can be found in Section \ref{section: assumptions}.

As we mentioned earlier, the goal of our paper is to prove the global existence of mild solutions for the mass-controlled, stochastically perturbed system \eqref{eq: initial mxr system} under Assumptions \ref{assump: initial data are >= 0 and in L^oo}--\ref{assump: Integrability of kernel L} on $u_0,f$, $\sigma$, and the noise $\dot{W}$. We will now make our goal rigorous. The mild solution to \eqref{eq: initial mxr system} with initial data $u_0 = (u_{10},...,u_{m0})$ is defined to be the solution $u=(u_1,..,u_m)$ of the integral equation
\begin{align}\label{eq: definition mild solution}
\begin{split}
    u_i(t,x) = \int_D G_i(t,x,y)u_0(y)dy &+ \int_0^t\int_D G_i(t-s,x,y)f_i(u(s,y))dyds\\
    &+\sum_{k=1}^r\int_0^t\int_D G_i(t-s,x,y)\sigma_{ik}(u(s,y))W_k(dyds),
\end{split}
\end{align}
where $G_i(t,x,y)$ is the heat kernel associated with the operator $d_i\Delta$ on $D$ with either Dirichlet or Neumann boundary conditions.

By \textit{global} existence of \eqref{eq: definition mild solution} we mean the following. Let us assume for a moment that $f$ and $\sigma$ are locally Lipschitz continuous in accordance with the later assumptions \ref{assump: locally Lipschitz f} and \ref{assump: Locally Lipschitz sigma} and that $\dot{W}$ is a Gaussian noise, white in time and colored in space (see Definition \ref{def: space time noise}). For any $n\in \mathbb N$, we may define localized functions $f_n$ and $\sigma_n$ such that for any $a\in \mathbb R^m$, $f_n(a)$ and $\sigma_n(a)$ match $f$ and $\sigma$ respectively when $\max\limits \{|a_1|,...,|a_m|\}\leq n$ and are globally Lipschitz continuous and bounded when $\max \{|a_1|,...,|a_m|\}>n$. There are many ways of realizing such a construction. One example is
\[
f_n(a) = \begin{cases}
    f(a),&\max\{|a_1|,...,|a_m|\}\leq n\\
    f\left(\frac{na}{\max\{|a_1|,...,|a_n|\}}\right), &\max\{|a_1|,...,|a_m|\}> n
\end{cases},
\]
and
\[
\sigma_n(a) = \begin{cases}
    \sigma(a),&\max\{|a_1|,...,|a_m|\}\leq n\\
    \sigma\left(\frac{na}{\max\{|a_1|,...,|a_n|\}}\right), &\max\{|a_1|,...,|a_m|\}> n
\end{cases},
\]
which Cerrai used in \cite{cerrai2003stochastic}. By construction, $f_n(a)$ and $\sigma_n(a)$ are globally Lipschitz. Therefore, by standard Picard iteration arguments (see, for instance, \cite{da2014stochastic} and \cite{dalang2026stochastic}) there exists a unique \textit{global} localized mild solution $u_n(t,x) = (u_{1,n}(t,x),...,u_{m,n}(t,x))$ to the problem
\begin{align}\label{eq: truncated global solution mxr}
\begin{split}
    u_{i,n}(t,x) = \int_D G_i(t,x,y)u_0(y)dy &+ \int_0^t\int_D G_i(t-s,x,y)f_{i,n}(u_n(s,y))dyds\\
    &+\sum_{k=1}^r\int_0^t\int_D G_i(t-s,x,y)\sigma_{ik,n}(u_n(s,y))W_k(dyds),
\end{split}
\end{align}
namely, $u_n(t,x)$ is a solution of this problem for all times and $x\in D$. In fact, these $u_n(t,x)$ match with each other in the sense that for $m>n$,
\[
u_n(t,x) = u_m(t,x) \text{ for } x\in D,  0\leq t\leq \tau_n := \inf\Big \{t>0: \sup\limits_{x\in D}\sup\limits_{i=1,...,m} |u_{i,n}(t,x)| \geq n\Big\}.
\]
This observation allows us to define the unique \textit{local} mild solution of \eqref{eq: initial mxr system} to be 
\[
u(t,x) := u_n(t,x) \text{ for all } x\in D \text{ and } t\in [0,\tau_n].
\]
We then consider the \textit{existence time of the local solution}
\[
\tau_{\infty}:= \sup_{n\in \mathbb N} \tau_n.
\]
This object is well defined since the sequence $\{\tau_n\}_{n\in \mathbb N}$ is non decreasing. We say that the local mild solution is a \textit{global} mild solution if it never explodes with probability one,  namely $\mathbb P(\tau_{\infty} = \infty) = 1$.

Proving $\mathbb P(\tau_{\infty} = \infty) = 1$ is the aim of this paper. We comment here briefly on the intricacy of this task. Standard results (see, for instance, Theorem 4.5.5. in \cite{dalang2026stochastic}) assert that the global solution $u_n$ of \eqref{eq: truncated global solution mxr} satisfies the $p$th moment bound
\begin{equation}\label{eq: pth moment bound}
\mathbb E \sup_{t\in [0,T]}\sup_{x\in D}\sup_{i=1,...,m} |u_{i,n}(t,x)|^p \leq C_{n,p,T},
\end{equation}
where the constant $C_{n,p,T}$ may depend on the Lipschitz constants of $f_n$ and $\sigma_n$. In Theorem \ref{thm: Theorem 2xr} (for the $2\times r$ system) and Theorem \ref{thm: Theorem mxr} (for the general $m\times r$ system) we will show that under our assumptions, these moment bounds do not depend on the Lipschitz constants of $f_n$ and $\sigma_n$ and that the moment bounds are in fact uniform with respect to $n$,
\begin{equation}\label{eq: uniform bound}
\sup_{n\in \mathbb N} \mathbb E \sup_{t\in [0,T]}\sup_{x\in D}\sup_{i=1,...,m} |u_{i,n}(t,x)|^p < \infty.
\end{equation}
As soon as the uniform in $n$ estimate \eqref{eq: uniform bound} is established, the proof that $\mathbb P(\tau_{\infty} = \infty)=1$ is a straightforward consequence of Markov's inequality, since for any fixed time horizon $T>0$,
\[
\mathbb P\left( \tau_n\leq T\right) =\mathbb P\left( \sup\limits_{t\in [0,T]}\sup_{x\in D}\sup\limits_{i=1,...,m} u_{i,n}(t,x)\geq n\right) \leq \frac{\mathbb E\sup\limits_{t\in [0,T]}\sup\limits_{x\in D}\sup\limits_{i=1,...,m} |u_{i,n}(t,x)|^p}{n^p} \xlongrightarrow[n\rightarrow \infty]{} 0.
\]

In order to prove \eqref{eq: uniform bound}, we take advantage of the triangular structure \eqref{eq: TMC} of the reactions $f$. This handling of the reactions is novel in the sense that all currently existing arguments proving the global existence of stochastic reaction--diffusion systems rely on strong dissipativity of the reactions to construct a contraction mapping and conclude existence via a fixed point argument. What we demonstrate is that even in the absence of strong dissipativity, we can use the structural interplay between the reactions to close the proof without invoking a fixed point argument.

So, with the aim of  \eqref{eq: uniform bound} in mind we have structured the paper as follows. In Section \ref{section: assumptions} we state the assumptions of our problem. In Section \ref{section: preliminary results} we collect the auxiliary results that will assist us throughout the proofs of Section \ref{section: Main results}. Finally, in Section \ref{section: Main results} we prove \eqref{eq: uniform bound} (and thus global existence) first for the $2\times r$ system and then for the initial $m\times r$ system by virtue of induction. We mention here for the convenience of the reader that \eqref{eq: truncated global solution mxr} will be the main object of interest throughout the paper and will be referred to frequently.
\section{Notational remarks}
In what follows, we set $Q_t := (0,t)\times D$, where $D\subset \mathbb R^d$ is an open and bounded domain with sufficiently smooth boundary. By $||\cdot||_2$ we mean the Euclidean norm of vectors. With $||\cdot||_{L^{\infty}(D)}$ we denote the usual supremum norm
\[
||v||_{L^{\infty}(D)} = \operatorname*{ess\,sup}_{x \in D} |v(x)|.
\]
For any $p\geq 1$, we will use the $L^p$ norms
\[
||g||_{L^p(Q_t)} := \Bigg (\int_0^t\int_D |g(s,x)|^pdxds\Bigg)^{1/p},\quad ||g(s)||_{L^p(D)} := \Bigg(\int_D |g(s,x)|^pdx\Bigg)^{1/p},
\]
where $s > 0$ is a time variable. If $\varepsilon \in (0,1)$ and $D\subset \mathbb R^d$, we denote by $||\cdot||_{W^{\varepsilon,p}(D)}$ the fractional Sobolev norm in $W^{\varepsilon,p}(D;\mathbb R)$,
\[
||v||_{W^{\varepsilon,p}(D)} := \Bigg(\int_D |v(x)|^p dx + \int_D\int_D \frac{|v(x)-v(y)|^p}{||x-y||_2^{d+\varepsilon p}}dxdy\Bigg)^{\frac{1}{p}}.
\]
Finally, for any $\theta\in (0,1)$, we denote by $C^{\theta}(\overline{D})$ the subspace of $\theta$-H\"older functions endowed with the norm
\[
\|v\|_{C^\theta(\overline D)}
:=
\|v\|_{L^\infty(D)} + \sup_{\substack{x,y \in \overline D \\ x \neq y}}
\frac{|v(x)-v(y)|}{||x-y||_2^\theta}.
\]
\section{Assumptions}\label{section: assumptions}
We make the following assumption on the initial data.
\begin{assumption}\label{assump: initial data are >= 0 and in L^oo}
    For each $i\in \{1,...,m\}$, $u_{i0}(x)\in L^{\infty}(D)$ and $u_{i0}(x)\geq 0$.
\end{assumption}
\subsection{Assumptions on $f$}
We assume the following for $f(a) = \bigl(f_1(a),...,f_m(a)\bigr) :\mathbb R^m_{+} \rightarrow \mathbb R^m$.
\begin{assumption}[Locally Lipschitz]\label{assump: locally Lipschitz f}
    For every $R>0$ there exists $L_R>0$ such that, whenever $||x||_2,||y||_2\leq R$,
\[
|f_i(x)-f_i(y)|\leq L_{R}||x-y||_2,
\]
for all $i\in \{1,...,m\}$.
\end{assumption}
\begin{assumption}[Quasipositivity]\label{assump: Positivity f} For any $a_1,...,a_m\geq 0,$
    \begin{equation*}
     f_i(a_1,...,a_{i-1},0,a_{i+1},...,a_m) \geq 0,
\end{equation*}
for all $i\in \{1,...,m\}$.
\end{assumption}
\begin{assumption}[Triangular mass control]\label{assump: triangular mass control f}
    For any $a_1,...,a_m\geq 0$ there exist constants $C_1,...,C_m\in \mathbb R$ such that, if $a=(a_1,...,a_m)$ then
\begin{align}\label{eq: triangular structure}
\begin{split}
    &f_{1}(a)\leq C_1(1+a_1+...+a_m),\\
    &f_1(a)+f_2(a)\leq C_2(1+a_1+...+a_m),\\
    &\vdots\\
    &f_1(a)+...+f_m(a)\leq C_m(1+a_1+...+a_m).
\end{split}
\end{align}
\begin{remark}
    Let $P$ be the following $m\times m$, lower triangular matrix
    \[
P=\begin{pmatrix}
1 & 0 & 0 & \cdots & 0 \\
1 & 1 & 0 & \cdots & 0 \\
1 & 1 & 1 & \cdots & 0 \\
\vdots & \vdots & \vdots & \ddots & \vdots \\
1 & 1 & 1 & \cdots & 1
\end{pmatrix}.
\]
Then \eqref{eq: triangular structure} can be written as
\begin{equation}\label{eq: triang structure matrix form}
Pf(a)\leq \mathbf{C}\left(1+\sum_{i=1}^m a_i\right),
\end{equation}
where $\mathbf{C} = (C_1,...,C_m)$. In that way, we may generalize \eqref{eq: triangular structure} by replacing $P$ by any lower triangular invertible $m\times m$ matrix with nonnegative entries. For the sake of simplicity, we will use condition \eqref{eq: triangular structure} throughout the paper, although our arguments go through naturally for the more general case of \eqref{eq: triang structure matrix form} as well.
\end{remark}
\end{assumption}
\begin{assumption}[Polynomial growth]\label{assump: Polynomial growth f}
     There exists $C>0$ and $\mu>0$ such that for any $a = (a_1,...,a_m)\in [0,\infty)^m$ and $i\in \{1,...,m\}$,
\begin{equation}\label{f Polynomial Growth 2x2}
|f_i(a)| \leq C(1+a_1^\mu+...+a_m^\mu).
\end{equation}
\end{assumption}
\subsubsection{Examples of reactions $f$.}\label{subsubsec: Examples f} Following are some more examples of reactions our paper treats that lie outside the framework
of \cite{cerrai2003stochastic} and \cite{salins2021systemsLDP}. For the deterministic system, Assumption \ref{assump: Positivity f} guarantees that $u_1$ and $u_2$ remain nonnegative for all times and $x\in D$ (see \cite{pierre2010global}). Under stochastic perturbation, nonnegativity of solutions holds under additional structural assumptions on the noise coefficient. This is the subject of Lemma \ref{lemma: non negativity of solutions}. To showcase the following examples, we assume that $u_i$ are nonnegative.
\begin{enumerate}
    \item The following reactions correspond to the well-known  \textit{Brussellator system}, which models chemical morphogenic processes \cite{bruss1},\cite{bruss2}:
    \begin{align*}
        f_1(u_1,u_2) &= -u_1u_2^2+\beta u_2,\\
        f_2(u_1,u_2) &= u_1u_2^2-(\beta+1)u_2+\alpha,
    \end{align*}
    were $\alpha,\beta >0$. Then, $f_1(0,u_2) = \beta u_2\geq 0$, $f_2(u_1,0) = \alpha \geq 0$, verifying Assumption  \ref{assump: Positivity f}, and $f_1(u_1,u_2)+f_2(u_1,u_2) \leq -u_2+\alpha$, verifying Assumption \ref{assump: triangular mass control f}. For the polynomial growth assumption we use $ab\leq (a^2+b^2)/2$. Then, $|f_1(u_1,u_2)|\leq c_1(u_2+u_1^2+u_2^4)$ and $|f_2(u_1,u_2)|\leq c_2(1+u_1^2+u_2^4)$, verifying Assumption \ref{assump: Polynomial growth f}.
    \item The following reactions provide a model for diffusive calcium dynamics \cite{application1}:
    \begin{align*}
f_1(u_1,u_2,u_3,u_4) &=
(1 + u_4)(1 - u_1)
- \frac{u_1^4}{1 + u_1^4}, \\[0.5em]
f_2(u_1,u_2,u_3,u_4) &= 
- u_2 + u_3, \\[0.5em]
f_3(u_1,u_2,u_3,u_4) &=
- (1 + u_1) u_3
+ u_2 + u_4, \\[0.5em]
f_4(u_1,u_2,u_3,u_4) &=u_1 u_3 + u_5
- (1+u_1) u_4, \\[0.5em]
f_5(u_1,u_2,u_3,u_4)&=
u_1 u_4 -  u_5.
\end{align*}
\end{enumerate}
Examples of such physical systems abound. We refer the curious reader to Pierre's survey \cite{pierre2010global} and the references therein for a more comprehensive list of such models. Quite recently,  reactions satisfying very similar assumptions (but very different boundary conditions) were used to model, interestingly enough, the early-stage spatial spread of Amyloid-$\beta$ oligomers in Alzheimer's disease \cite{alzheimers}.
\subsection{Assumptions on $\sigma$}
We assume the following for $\sigma(a) = [\sigma_{ik}(a)]_{i\in \{1,...,m\},k\in \{1,...,r\}}: \mathbb R^{m}_{+}\rightarrow \mathbb R^{m\times r}$.
\begin{assumption}[Locally Lipschitz]\label{assump: Locally Lipschitz sigma} For every $R>0$ there exists $L_\sigma>0$ such that, whenever $||x||_2,||y||_2\leq R$,
\[
|\sigma_{ik}(x) - \sigma_{ik}(y)|\leq L_R ||x-y||_2,
\]
for any $i\in\{1,...,m\}$ and $k\in \{1,...,r\}$.
\end{assumption}
\begin{assumption}[Positivity]\label{assump: Positivity sigma} For any $k\in \{1,...,r\}$ and $a_1,...,a_m\geq 0$,
\[
\sigma_{ik}(a_1,...,a_{i-1},0,a_{i+1},...,a_m) = 0,
\]
for any $i\in \{1,...,m\}$.
\end{assumption}
\begin{assumption}[Linear Growth]\label{assump: Linear growth sigma} For any $i\in \{1,...,m\}$,  $k\in \{1,...,r\}$ and $a = (a_1,...,a_m)\in [0,\infty)^m$,
\[
|\sigma_{ik}(a)| \leq C(1+a_1+...+a_m).
\]
\end{assumption}
\subsection{Assumptions on the noise $\dot{W}$}
We first provide the definition of a Gaussian space-time noise that is white in time and colored in space.
\begin{definition}\label{def: space time noise}
    Let $L: D\times D\rightarrow \mathbb R$ be a symmetric, positive--definite generalized function. A white-in-time, colored-in-space Gaussian noise with spatial covariance kernel $L$ is a centered Gaussian random measure on $[0,T]\times D$ such that for all adapted test functions $\varphi,\psi\in C^{\infty}_c([0,T]\times D)$,
\begin{align*}\label{eq: space time noise}
&\mathbb{E}
\left(\int_0^T\int_D \varphi(s,x)\,W(dxds)\right)
\left(\int_0^T\int_D \psi(s,y)\,W(dyds)\right)
\\
=&\mathbb E\int_0^T\int_D\int_D
\varphi(s,x)\,\psi(s,y)L(x,y)dxdyds.
\end{align*}
\end{definition}
\noindent 
In our paper, we care about the case where $W=(W_1,...,W_r)$ is a vector of white-in-time, colored-in-space Gaussian noises whose components satisfy Definition \ref{def: space time noise} with spatial covariance kernels $L_1(x,y),...,L_r(x,y)$ that adhere to the following assumptions:
\begin{assumption}[Positivity]\label{assump: Posivitiy of kernel L}
    For any $k\in \{1,...,r\}$, $L_k$ is symmetric, positive definite, and pointwise nonnegative in $D\times D$.
\end{assumption} 
\begin{assumption}[Heat kernel singularity]\label{assump: heat kernel singularity}
    For any $i\in \{1,...,m\}$, there exists a $C>0$ such that 
   \begin{equation}\label{eq: heat kernel singularity}
       \sup_{x\in D}\sup_{y\in D} G_i(t,x,y) \leq Ct^{-\frac{d}{2}},
   \end{equation}
   where $d\geq 1$ is the spatial dimension.
\end{assumption}
\begin{assumption}[Heat kernel convolution singularity]\label{assump: convolution singularity}
    For any $i\in \{1,...,m\}$ and $k\in \{1,...,r\}$, there exist constants $C>0$ and $\eta\in(0,1)$ such that for all $t\in(0,T]$,
\begin{equation}\label{eq:L2}
\sup_{x\in D}\int_D\int_D G_i(t,x,y_1)G_i(t,x,y_2)L_k(y_1,y_2)\,dy_1dy_2
\le C t^{-\eta}.
\end{equation}
\end{assumption}
\begin{assumption}[Integrability]\label{assump: Integrability of kernel L} For any $k\in \{1,...,r\}$, the kernels $L_k: D\times D\rightarrow \mathbb R$ are integrable in the sense that
    \[
    \sup_{y_1\in D}\int_{D} L_k(y_1,y_2)dy_2<\infty.
    \]
\end{assumption}
\noindent 
These assumptions are satisfied, for instance, by the Riesz kernel,
\[
L(x,y) = \frac{1}{||x-y||_2^{\beta}},\quad x,y\in D, \quad 0<\beta<d,
\]
and also by spectral kernels of the form
\[
L(x,y) = \int_0^{\infty} s^{\gamma - 1}e^{-\theta s}G(s,x,y)ds,
\]
for any $\theta > 0$ and for $\gamma < d/2$. As the authors in \cite{salins2025nonexplosion} point out, spectral kernels of this form can always be rewritten equivalently as
\[
L(x,y) = \sum_{j=1}^{\infty} \lambda_j^2e_j(x)e_j(y),
\]
Such spectral kernels were considered, for instance, by Da Prato and  Zabczyk in \cite{da2014stochastic}, and Cerrai in \cite{cerrai2003stochastic} and \cite{cerrai2009khasminskii}. For the verification of Assumptions \ref{assump: Posivitiy of kernel L}-\ref{assump: Integrability of kernel L} we refer to the Examples section of \cite{salins2025nonexplosion}. We note here that our Assumption \ref{assump: Integrability of kernel L} is slightly stronger than Assumption 3(B) of \cite{salins2025nonexplosion}, which requires $L\in L^1(D\times D)$.
\section{Auxiliary results}\label{section: preliminary results}
We gather here several results that are crucial for the proofs of Theorems \ref{thm: Theorem 2xr} and \ref{thm: Theorem mxr} in the next section. We begin with a Lemma about the nonnegativity of solutions to systems of equations. This lemma is a natural extension of Corollary 2.6. by Kotelenez \cite{kotelenez1992comparison}.
\begin{lemma}\label{lemma: non negativity of solutions}
Fix $i\in \{1,...,m\}$ and $k\in \{1,...,r\}$. Let $b_i:\mathbb R_{+}^m\to \mathbb R$ and $g_{ik}: \mathbb R^m_{+}\to \mathbb R$. Assume that:
    \begin{enumerate}
        \item $u_i$ solves the integral equation
        \begin{align*}
        u_i(t,x) &= \int_D G_i(t,x,y)u_{i0}(y)+\int_0^t\int_D G_i(t-s,x,y)b_i(u(s,y))dyds\\
    &+ \sum_{k=1}^r \int_0^t\int_D G_i(t-s,x,y)g_{ik}(u(s,y))W_k(dyds),
        \end{align*}
        where $G_i(t,x,y)$ is the heat kernel;
        \item The functions $b_i$ and $g_{ik}$ are globally Lipschitz continuous: There exists an $L>0$ such that, whenever $x,y\in [0,\infty)^m$,
        \begin{align*}
            &|b_i(x)-b_i(y)|\leq L||x-y||_2\\
            &|g_{ik}(x)-g_{ik}(y)|\leq L||x-y||_2
        \end{align*}
        \item The functions $b_i$ and $g_{ik}$ are uniformly bounded: There exists a $K>0$ such that
        \begin{align*}
            &\sup_{\substack{a_1,...,a_m\\ a_i\geq 0}} |b_i(a_1,...,a_m)|\leq K,\\
            &\sup_{\substack{a_1,...,a_m\\ a_i\geq 0}} |g_{ik}(a_1,...,a_m)|\leq K.
        \end{align*}
        \item The functions $b_i$ and $g_{ik}$ satisfy the positivity conditions
        \begin{align*}
        b_i(a_1,...,a_{i-1},0,a_{i+1},...,a_m) &\geq 0;\\
        g_{ik}(a_1,...,a_{i-1},0,a_{i+1},...,a_m) &= 0, \quad \text{ for any } k\in \{1,...,r\},
        \end{align*}
        whenever $a_1,...,a_m\geq 0$.
        \item The initial data $u_{i0}(y)$ are nonnegative a.e.
    \end{enumerate}
    Then $\mathbb{P}\bigl(u_i(t,\cdot)\ge 0 \text{ a.e. in } D \text{ for all } t\in[0,T]\bigr)=1$.
\end{lemma}
\begin{proof}
    We define the new  functions 
    \begin{align*}
    \tilde{b}_1(t,x,z) &:= b_1(z,u_2(t,x)\wedge 0,...,u_m(t,x)\wedge 0),\\
    \tilde{g}_1(t,x,z) &:= g_1(z,u_2(t,x)\wedge 0,...,u_m(t,x)\wedge 0),
    \end{align*}
     Then $u_1(t,x)$ solves the decoupled integral equation
    \begin{align*}
        u_1(t,x) &= \int_D G_1(t,x,y)u_{10}(y)+\int_0^t\int_D G_1(t-s,x,y)\tilde{b}_1(s,y,u_1(s,y))dyds\\
    &+ \sum_{k=1}^r \int_0^t\int_D G_1(t-s,x,y)\tilde{g}_{1k}(s,y,u_1(s,y))W_k(dyds),
        \end{align*}
 for which the assumptions of Corollary 2.6 by Kotelenez \cite{kotelenez1992comparison} are satisfied. Therefore, $u_1(t,x)\geq 0$ for all $t\in [0,T]$ and for almost all $x\in D$. By defining
 \begin{align*}
     \tilde{b}_i(t,x,z) &:= b_i(u_1(t,x)\wedge 0,...,u_{i-1}(t,x)\wedge 0,z,u_{i+1} (t,x)\wedge 0,...,u_m(t,x)\wedge 0),\\
     \tilde{g}_{ik}(t,x,z) &:= g_{ik}(u_1(t,x)\wedge 0,...,u_{i-1}(t,x)\wedge 0,z,u_{i+1}(t,x)\wedge 0,...,u_m(t,x)\wedge 0),
 \end{align*}
 an identical argument yields the non negativity of $u_2,...,u_m$ 
\end{proof}
Consequently, since by construction $f_{i,n}$ and $\sigma_{ik,n}$ satisfy the assumptions of Lemma \ref{lemma: non negativity of solutions}, we obtain the nonnegativity of $u_{i,n}$ as defined in \eqref{eq: truncated global solution mxr}.
\begin{corollary}\label{corollary: nonnegativity of u_i,n}
    We have that $u_{i,n}(t,x)\geq 0$ almost surely for all $t\in [0,T]$ and almost all $x\in D$.
\end{corollary}
 \subsection{H\"older continuity and estimates on the stochastic convolution}
We turn our attention now to some estimates regarding the stochastic part of \eqref{eq: truncated global solution mxr}.
\begin{proposition}\label{proposition: Holder regularity of Z_ik}
Fix $i\in \{1,...,m\}$ and $k\in \{1,...r\}$ and let $Z_{ik}$ be the stochastic convolution
\begin{equation}\label{eq: Z_ik}
Z_{ik} := \int_0^t\int_D G_i(t-s,x,y)g_{ik}(s,y)W_k(dyds),
\end{equation}
where $g_{ik}:\mathbb R\rightarrow \mathbb R$ is bounded and adapted and $G_i$ is the heat kernel corresponding to the operator $d_i\Delta$. Then, for every $t\in [0,T]$,
\[
\mathbb E\sup_{t\in [0,T]}||Z_{ik}(t)||^p_{W^{\varepsilon,p}(D)}\leq C_{i,k,p,\varepsilon,T}||g_{ik}||_{L^{\infty}(D)}^p,
\]
for $\varepsilon<1-\eta$, where $\eta$ is given in Assumption \ref{assump: convolution singularity}, and for sufficiently large $p$ depending on $\varepsilon$ and $\eta$.
\end{proposition}
\begin{proof}
    We make use of the factorization method by Da Prato and Zabczyk (Theorem 5.10 in \cite{da2014stochastic}). In particular, let $\alpha \in (0,1)$ and define
\[
Z_{\alpha,ik}(t,x):= \int_0^t \int_D (t-s)^{-\alpha} G_i(t-s,x,y)g_{ik}(s,y)W_{k}(dy ds),
\]
then
\[
Z_{ik}(t,x) = \frac{\sin( \pi\alpha)}{\pi} \int_0^t \int_D (t-s)^{\alpha-1} G_i(t-s,x,y)Z_{\alpha,ik}(s,y)dyds.
\]
We apply the BDG inequality,
\begin{align*}
\mathbb{E}\lvert Z_{\alpha,ik}(t,x) \rvert^p
&\le C_p\, \mathbb{E}\Bigg(
\int_0^t \int_D\int_D (t-s)^{-2\alpha}
G_i(t-s,x,y_1) G_i(t-s,x,y_2) \\
&\qquad\qquad\qquad\qquad
\times g_{ik}(s,y_1) g_{ik}(s,y_2)
L_k(y_1,y_2)
dy_1dy_2ds
\Bigg)^{\frac{p}{2}}.
\end{align*}
We use that $g_{ik}$ is bounded, $L_k\geq 0$ (Assumption \ref{assump: Posivitiy of kernel L}), $G_i\geq 0$, and that the heat kernel satisfies the singularity estimate of Assumption \ref{assump: convolution singularity}, so that for fixed $t \in [0,T]$, and $x \in D$,
\begin{align*}
\mathbb E|Z_{\alpha,ik}(t,x)|^p&\le C_p \|g_{ik}\|_\infty^p
\left(
\int_0^t \int_D\int_D
(t-s)^{-2\alpha}G_i(t-s,x,y_1) G_i(t-s,x,y_2)\,
L_k(y_1,y_2)\,
dy_1\, dy_2\, ds
\right)^{\frac{p}{2}} \\
&\le C_{i,k,p} \|g_{ik}\|_\infty^p
\left(
\int_0^t (t-s)^{-2\alpha - \eta}\, ds
\right)^{\frac{p}{2}}.
\end{align*}
If $\alpha < \frac{1-\eta}{2}$, the last integral is finite and
\begin{align*}
\mathbb E|Z_{\alpha,ik}(t,x)|^p&\le C_{i,k,p,\alpha}\, \|g_{ik}\|_\infty^p\, t^{(1-\eta-2\alpha)\frac{p}{2}}.
\end{align*}
Thus
\begin{equation}\label{eq: facorization L^p bound}
\mathbb{E}\|Z_{\alpha,ik}(t)\|_{L^p(D)}^p
= \mathbb{E}\int_D |Z_{\alpha,ik}(t,x)|^p\, dx
\le |D|C_{i,k,p,\alpha}\, t^{(1-\eta-2\alpha)\frac{p}{2}}.
\end{equation}
Now, let $S_i(t)$ be the semigroup associated with the heat kernel $G_i$. Then we may represent $Z_{ik}(t)$ as
\begin{align*}
Z_{ik}(t)
&= C_{\alpha} \int_0^t (t-s)^{\alpha-1} S(t-s)\, Z_{\alpha,ik}(s)\, ds.
\end{align*}
And then
\begin{align*}
\|Z_{ik}(t)\|_{W^{\varepsilon,p}(D)}
&\le C_{\alpha} \int_0^t (t-s)^{\alpha-1}
\|S_i(t-s) Z_{\alpha,ik}(s)\|_{W^{\varepsilon,p}(D)}\, ds.
\end{align*}
For any $t\in [0,T]$ and $\varepsilon>0$, the semigroup $S(t)$ maps $L^p(D)$ into $W^{\varepsilon,p}(D)$ and by estimate (2.4) in \cite{cerrai2003stochastic} we get
\begin{align*}
||Z_{ik}(t)||_{W^{\varepsilon,p}(D)}&\le C_{i,k} \int_0^t (t-s)^{\alpha-1-\frac{\varepsilon}{2}}
\|Z_{\alpha,ik}(s)\|_{L^p}\, ds.
\end{align*}
Then, by H\"older's inequality
\begin{align*}
&||Z_{ik}(t)||_{W^{\varepsilon,p}(D)}\le C_{i,k,\varepsilon}
\left(
\int_0^t (t-s)^{(\alpha-1-\frac{\varepsilon}{2})\frac{p}{p-1}} ds
\right)^{\frac{p-1}{p}}
\left(
\int_0^t \|Z_{\alpha,ik}(s)\|_{L^p(D)}^pds
\right)^{\frac{1}{p}}.
\end{align*}
For the first integral to be finite we require
\[
 p > \frac{1}{\alpha - \frac{\varepsilon}{2}} \quad \text{and} \quad \frac{\varepsilon}{2} < \alpha.
\]
Then,
\begin{align*}
\|Z_{ik}(t)\|_{W^{\varepsilon,p}(D)}
&\le C_{i,k,p,\varepsilon,\alpha}\,
t^{\alpha - 1 - \frac{\varepsilon}{2} + \frac{p-1}{p}}
\left(
\int_0^t \|Z_{\alpha,ik}(s)\|_{L^p(D)}^p\, ds
\right)^{\frac{1}{p}}.
\end{align*}
So, we obtain
\begin{align*}
\mathbb{E} \sup_{t\in [0,T]} \|Z_{ik}(t)\|_{W^{\varepsilon,p}(D)}^p
&\le C_{i,k,p,\varepsilon,\alpha}\,
T^{\alpha p - \frac{\varepsilon p}{2} - p + 1}\,
\mathbb{E}
\int_0^T \int_D |Z_{\alpha,ik}(s,y)|^pdyds.
\end{align*}
Therefore, by \eqref{eq: facorization L^p bound}
\begin{align*}
\mathbb E\sup_{t\in [0,T]}||Z_{ik}(t)||_{W^{\varepsilon,p}(D)}^p \leq C_{i,k,p,\varepsilon,\alpha}T^{2-p+\frac{(1-\eta-\varepsilon)p}{2}}||g_{ik}||_{\infty}^p,
\end{align*}
which, as we remarked earlier, requires $\frac{\varepsilon}{2} < \alpha \text{ and } p>\frac{1}{\alpha - \frac{\varepsilon}{2}}$. And since we have already required $\alpha < \frac{1-\eta}{2}$, we conclude that
\[
Z_{ik}(t,\cdot) \in W^{\varepsilon,p}(D) \quad \text{ for any } \varepsilon<1-\eta, \text{ for } p \text{ large enough }.
\]
\end{proof}
\begin{remark}
If we choose $p$ sufficiently large in the above Proposition so that $\varepsilon p>d$, then by the Sobolev
embedding theorem (see, for instance, Theorem 6(ii) of Section 5.6.3 in \cite{evans2022partial}),
\[
W^{\varepsilon,p}(D)\hookrightarrow C^{\varepsilon-\frac{d}{p}}(\overline D),
\]
and hence
\[
Z_{ik}(t,\cdot)\in C^{\varepsilon-\frac{d}{p}}(\overline D).
\]
Since $p$ can be taken arbitrarily large and we assumed $\varepsilon < 1-\eta$, we conclude that
\begin{equation}\label{eq: Z(t,dot) Holder for any theta}
Z_{ik}(t,\cdot)\in C^{\theta}(\overline D)
\quad\text{for any } \theta < 1-\eta.
\end{equation}
\end{remark}
We can also make a H\"older regularity statement about the deterministic integral, as the next lemma shows.
\begin{lemma}\label{Lemma: Holder regularity of deterministic integral}
    If $f_i(t,x)\in L^{\infty}(Q_1)$, then
    \[
    \int_0^t\int_D G_i(t-s,x,y)f_i(s,y)dyds\in C^{\theta}(\overline{D}).
    \]
\end{lemma}
\begin{proof}
    Applying the smoothing estimate (2.5) in \cite{cerrai2003stochastic}, we see that
    \[    ||S_i(t)f_i||_{C^{\theta}(\overline{D})} \leq Ct^{-\frac{\theta}{2}}||f_i||_{L^{\infty}(D)},
    \]
    for any $0<\theta<1-\eta$. Therefore,
\begin{equation}\label{eq: determ smoothing}
\Bigg|\Bigg| \int_0^t S_i(t-s)f_i(s)ds\Bigg|\Bigg|_{C^{\theta}(\overline{D})} \leq C||f_i||_{L^{\infty}(D)}\int_0^t (t-s)^{-\frac{\theta}{2}}ds<\infty.
\end{equation}
\end{proof}
Returning back to $u_{i,n}$ in \eqref{eq: truncated global solution mxr}, and recalling that $f_{i,n}(u_n(s,y))$ was bounded and $\sigma_{ik,n}(u_n(s,y))$ was bounded and adapted, we obtain the following Corollary that establishes the H\"older continuity of $u_{i,n}$.
\begin{corollary}[H\"older continuity]
    For the truncated global solutions as defined in \eqref{eq: truncated global solution mxr}, $u_{i,n}(t,\cdot)\in C^{\theta}(\overline{D})$ almost surely for every $t>0$ and for any $\theta<1-\eta$, where $\eta$ was given in Assumption \ref{assump: convolution singularity}.
\end{corollary}
\begin{proof}
    The proof is a direct consequence of Proposition \ref{proposition: Holder regularity of Z_ik} and Lemma \ref{Lemma: Holder regularity of deterministic integral}.
\end{proof}
The next lemma was proved in \cite{salins2025nonexplosion} for a single stochastic convolution. A trivial modification of the proof allows one to obtain the following for \eqref{eq: Z_ik}.
\begin{lemma}[Theorem 1 in \cite{salins2025nonexplosion}]\label{lemma: sup norm Z}
    Let $p$ be large enough such that $\frac{d+2}{p} + \frac{d}{p-2} < 1-\eta$, where $\eta$ is given in Assumption \ref{assump: convolution singularity}, and consider the stochastic convolution as defined in \eqref{eq: Z_ik}. Then, there exists a constant $C_p>0$, independent of $T>0$ and $g_{ik}$, such that
\begin{equation}\label{eq: sup norm Z}
\mathbb{E}\sup_{(t,x)\in Q_T}\sup_{i}\bigl|Z_{ik}(t,x)\bigr|^p\leq C_p
T^{\frac{(1-\eta)(p-2)}{2}-d}\,
\mathbb{E}
\int_0^T\int_D\int_D \bigl|g_{ik}(s,y_1)g_{ik}(s,y_2)\bigr|^{\frac{p}{2}} L_k(y_1,y_2)dy_1dy_2ds.
\end{equation}
\end{lemma}
We now make some useful observations. To this end, we set
    \begin{equation}\label{eq: alpha}
a := \frac{(1-\eta)(p-2)}{2}-d.
    \end{equation}
Besides bounds on the supremum norm, the above Lemma will also serve us by providing direct $L^p(Q_t)$-norm bounds for $Z_{ik}$. Indeed, 
\[
\int_0^t\int_D |Z_{ik}(r,x)|^pdxdr\leq t|D|\sup_{r\in [0,t]}\sup_{x\in D} |Z_{ik}(r,x)|^p,
\]
and therefore
taking expectations and applying Lemma \ref{lemma: sup norm Z} gives
\begin{equation}\label{eq: Lp norm Z}
\mathbb E||Z_{ik}||_{L^p(Q_t)}^{p}
\leq C_p t^{a+1}
\mathbb E\int_0^t\int_D\int_D
|g_{ik}(s,y_1)g_{ik}(s,y_2)|^{\frac{p}{2}}
L_k(y_1,y_2)dy_1dy_2ds.
\end{equation}
Moreover, using the inequality $ab\leq \dfrac{1}{2}(a^2+b^2)$ and Assumption \ref{assump: Integrability of kernel L} we may write,
\begin{align*}
\int_D\int_D
|g_{ik}(s,y_1)g_{ik}(s,y_2)|^{\frac{p}{2}} L_k(y_1,y_2)dy_1dy_2&\leq C\int_D |g_{ik}(s,y)|^pdy.
\end{align*}
Therefore, by \eqref{eq: Lp norm Z} we can also obtain
\begin{equation}\label{eq: Z_ik bond with g_ik|^p}
\mathbb E||Z_{ik}||^p_{L^p(Q_t)} \leq C_pt^{a+1}\mathbb E \int_0^t\int_D |g_{ik}(s,y)|^p dyds.
\end{equation}
For instance, returning to \eqref{eq: truncated global solution mxr}, we set
\begin{equation}\label{eq: Z_i,n}
Z_{i,n}(t,x) := \sum_{k=1}^r\int_0^t\int_D G_i(t-s,x,y)\sigma_{ik,n}(u_n(s,y))W_k(dyds) = \sum_{k=1}^r Z_{ik,n}(t,x).
\end{equation}
Since $\sigma_{ik,n}(u_n(s,y))$ are bounded and adapted, an application of \eqref{eq: Z_ik bond with g_ik|^p} yields
\[
\mathbb E||Z_{ik,n}||^p_{L^p(Q_t)} \leq C_pt^{a+1}\mathbb E\int_0^t\int_D |\sigma_{ik,n}(u_n(s,y))|^pdyds.
\]
And since $\sigma_{ik,n}$ grows at most linearly (Assumption \ref{assump: Linear growth sigma}),
\begin{equation*}
\mathbb E||Z_{ik,n}||^p_{L^p(Q_t)} \leq C_pt^{a+1}\mathbb E\int_0^t \int_D\bigl(1+|u_{1,n}(s,y)|^p+...+|u_{m,n}|^p\bigr)dyds,
\end{equation*}
where the constant $C_p$ depends on the linear growth 
constant of $\sigma$. Setting
\[
Y_n(t) := ||u_{1,n}||^p_{L^p(Q_t)}+...+||u_{m,n}||^p_{L^p(Q_t)} = \sum_{i=1}^m ||u_{i,n}||_{L^p(Q_t)}^p,
\]
gives
\begin{equation*}\label{eq: Lp norm Z with Y}
    \mathbb E||Z_{ik,n}||^p_{L^p(Q_t)}\leq C_pt^{a+1}(t+\mathbb EY_n(t)).
\end{equation*}
Therefore,
\begin{equation}\label{eq: Lp norm Z_i}
\mathbb E||Z_{i,n}||^p_{L^p(Q_t)} \leq \sum_{k=1}^r \mathbb E||Z_{ik,n}||^p_{L^p(Q_t)}\leq C_pt^{a+1}(t+\mathbb EY_n(t)).
\end{equation}
An identical argument also yields
\begin{equation}\label{eq: sup norm Z_i}
   \mathbb E \sup_{t\in [0,T]}\sup_{x\in D}\sup_{i=1,...,m}|Z_{i,n}(t,x)|^p \leq C_{p,T}(1+\mathbb EY_n(T)).
\end{equation}
\subsection{A lemma from the theory of parabolic PDE}
We will call $v = (v_1,...,v_m)$ a \textit{regular function} if it satisfies a Dirichlet or Neumann boundary condition (namely, if $v_i=0$ or $\partial_n v_i = 0$ on $(t,x) \in (0,T)\times \partial D$) as well as the following regularity conditions:
\begin{equation}\label{eq: system classical solution}
\begin{aligned}
\text{(a)}\quad & v\in C([0,T); L^1(D)^m)\cap L^{\infty}([0,T-\tau]\times D)^m, \ \forall \tau\in (0,T), \\
\text{(b)}\quad & \forall k,j=1,\dots,d,\ \forall p\geq 1:\ 
\partial_tv,\partial_{x_k}v,\partial_{x_kx_j} v \in L^p((\tau,T-\tau)\times D)^m,\ \forall \tau\in (0,T).
\end{aligned}
\end{equation}
We recall the following Lemma from the theory parabolic PDEs, which has been popularized by Michel Pierre. For its proof, we refer to \cite{pierre1987global} as well as \cite{pierre2010global} for a more modern treatment.
\begin{lemma}[Lemma 3.4 in \cite{pierre2010global}]\label{lemma: Duality lemma}
Let $v_1,v_2$ be regular functions in the sense of \eqref{eq: system classical solution} satisfying
\begin{equation}\label{heat equations inequality}
\partial_t v_1 - d_1 \Delta v_1 + \theta_1 v_1 
\le \theta_2 \partial_t v_2 - d_2 \Delta v_2 + \theta_3 v_2 + H
\end{equation}
together with the same constant Neumann or Dirichlet boundary conditions for $v_1,v_2$, where $\theta_i \in \mathbb{R}$ and $H \in L^p(Q_T)$, $H \ge 0$. Then, there exists $C>0$ depending on $p,d_1,d_2$ and the domain $D$ such that, for all $1<p<\infty$, $t\in (0,T]$, 
and with bounded initial data,
\begin{equation}\label{eq: pierre's lemma eq}
\|v_1^+\|_{L^p(Q_t)} \leq C \left( \|v_2\|_{L^p(Q_t)} + \|v_1(0)\|_{L^p(D)} + |\theta_2|  \|v_2(0)\|_{L^p(D)} + \int_0^t \|H(s)\|_{L^p(D)} ds \right).
\end{equation}
\end{lemma}
\begin{remark}
    The original statement of Lemma 3.4. in \cite{pierre2010global} differs from \eqref{eq: pierre's lemma eq} in the sense that the initial data have been absorbed into the constant $C$. However, for our later arguments it is important to keep track of the dependence on the initial data; therefore we make this dependence explicit. The proof of \eqref{eq: pierre's lemma eq} is completely identical to Pierre's in \cite{pierre2010global} and for this reason we omit it.
\end{remark}
\begin{remark}\label{Remark beneath duality lemma}
    If we replace \eqref{heat equations inequality} with
\begin{equation*}
\partial_t v_1 - d_1\Delta v_1 + \theta_1 v_1 \leq \Big( \theta_2\partial_t v_2 -d_2\Delta v_2 + \theta_3 v_2\Big)+\Big(\theta_4\partial_t v_3 - d_3\Delta v_3 + \theta_5 v_3\Big) + H,
\end{equation*}
we will equivalently get the bound,
\begin{align*}
\|v_1^+\|_{L^p(Q_t)}
&\le
C\Big(
\|v_2\|_{L^p(Q_t)}
+\|v_3\|_{L^p(Q_t)}
+\|v_1(0)\|_{L^p(D)}\\
&\qquad+|\theta_2|\,\|v_2(0)\|_{L^p(D)}+|\theta_4|\|v_3(0)\|_{L^p(D)}+\int_0^t \|H(s)\|_{L^p(D)}\,ds
\Big).
\end{align*}
In general, if we replace \eqref{heat equations inequality} with the more general inequality
\[
(\partial_t - d_1\Delta + \theta_1)v_1 \leq \sum\limits_{j=2}^m (\theta_j\partial_t  - d_j\Delta + \theta_j )v_{j}+H,
\]
we get the bound
\[
\|v_1^+\|_{L^p(Q_t)} \le C \left(
\sum_{j=2}^{m} \|v_j\|_{L^p(Q_t)}
+ \|v_1(0)\|_{L^p(\Omega)}
+ \sum_{j=2}^{m} |\theta_j| \, \|v_j(0)\|_{L^p(\Omega)}
+ \int_0^t \|H(s)\|_{L^p(\Omega)} \, ds
\right).
\]
\end{remark}
In view of the truncated global solution \eqref{eq: truncated global solution mxr}, we define $v_{i,n}(t,x):=u_{i,n}(t,x) - Z_{i,n}(t,x)$, where $Z_{i,n}$ is given by \eqref{eq: Z_i,n}. Namely,
\begin{equation}\label{eq: v_i,n definition}
v_{i,n}(t,x):=\int_D G_i(t,x,y)u_{i0}(y)dy + \int_0^t\int_D G_i(t-s,x,y)f_{i,n}(u_n(s,y))dyds.
\end{equation}
It is in the $v_{i,n}$ that we want to apply Lemma \ref{lemma: Duality lemma} later and therefore we must verify that $v_{i,n}$ is a regular function in the sense of \eqref{eq: system classical solution}. First, we notice that $v_{i,n}$ weakly solves the (random) PDE
\begin{equation}\label{eq: sys v_i,n}
\begin{cases}
    \partial_t v_{i,n}(t,x) - d_i\Delta v_{i,n}(t,x) = f_{i,n}(u_{n}(t,x)) \text{ in } (0,T)\times D,\\
    v_{i,n}=0 \text{ or } \dfrac{\partial v_{i,n}}{\partial n} = 0  \text{ on } (0,T)\times \partial D,\\
    v_{i,n}(0,x) = u_{i0}(x).
\end{cases}
\end{equation}
Since $u_{i0}(x)\in L^{\infty}(D)$ (Assumption \ref{assump: initial data are >= 0 and in L^oo}) and $f_{i,n}$ is bounded by construction, $v_{i,n}\in L^{\infty}(D)$ by \eqref{eq: v_i,n definition}. Furthermore, we may write \eqref{eq: v_i,n definition} in semigroup notation
\[
v_{i,n}(t) = S_i(t)u_{i0} + \int_0^t S_i(t-s)F_{i,n}(s)ds,
\]
where $S_i$ is the heat semigroup with kernel $G_i$ and $(F_{i,n}(s))(y) := f_{i,n}(u_n(s,y))$. Notice that $S_i(t)$ is a strongly continuous semigroup  in $L^1(D)$ as it is generated by the heat kernel. Therefore, $u_0\in L^{\infty}(D)\subset L^1(D)$, and $F_{i,n}$ is bounded above by a constant. In particular,
$F_{i,n}(s)\in L^\infty(D) \subset L^1(D)$. Thus $v_{i,n}(t)$ is continuous in $L^1(D)$ on $[0,T)$, which verifies part (a) of the characterization of a regular function. 

To verify part (b), notice that $F_{i,n}(s)\in W^{\varepsilon,p}(D)$ for all $s\in [0,T]$ and $\varepsilon\in (0,1-\eta)$. This is because $a\mapsto f_{i,n}(a)$ is globally Lipschitz continuous and $u_{i,n}(t,\cdot) \in W^{\varepsilon,p}(D)$ by combining Proposition \ref{proposition: Holder regularity of Z_ik} and Lemma  \ref{Lemma: Holder regularity of deterministic integral}. Also, if we denote by $D(A_i)$ the domain of $A_i = -d_i\Delta$, then, since $A_i$ is a second--order, uniformly elliptic operator with constant coefficients on a sufficiently regular domain, Theorem 3.1.1(ii) of \cite{lunardi2012analytic} applies: The inclusion $D(A_i)\hookrightarrow W^{2,p}(D)$ is continuous, with the norm on $D(A_i)$ being the graph norm, and we may write
\begin{align*}
    \Bigg \| \int_0^t S_i(t-s)F_{i,n}(s)ds \Bigg \|_{W^{2,p}(D)}&\leq \int_0^t ||S_i(t-s)F_{i,n}(s)||_{W^{2,p}(D)}ds\\
    &\leq C_p\int_0^t ||S_i(t-s)F_{i,n}(s)||_{L^p(D)}ds\\
    &+C_p\int_0^t ||A_iS_i(t-s)F_{i,n}(s)||_{L^p(D)}ds.
\end{align*}
Since $(S_i(t))_{t\ge0}$ is an analytic semigroup on $L^p(D)$, it is bounded. If we denote its boundedness constant by $M$, we obtain
\[
\int_0^t \|S_i(t-s)F_{i,n}(s)\|_{L^p(D)}\,ds
\le M t\sup_{0\le s\le t}\|F_{i,n}(s)\|_{L^p(D)} < \infty.
\]
For the second term, fix $\varepsilon\in(0,1)$. We write
\[
A_iS_i(t-s)=(-A_i)^{1-\frac{\varepsilon}{2}}S_i(t-s)(-A_i)^{\frac{\varepsilon}{2}}.
\]
Hence, by Theorem 6.13(c) of \cite{pazy2012semigroups}, there exists $C_{\varepsilon}>0$ such that
\[
\|(-A_i)^{1-\frac{\varepsilon}{2}}S_i(t-s)\|_{\mathcal L(L^p(D))}
\le C_{\varepsilon} (t-s)^{-(1-\frac{\varepsilon}{2})}.
\]
Consequently,
\begin{align*}
\int_0^t \|A_i S_i(t-s)F_{i,n}(s)\|_{L^p(D)}\,ds
&\le \int_0^t \|(-A_i)^{1-\frac{\varepsilon}{2}}S_i(t-s)\|_{\mathcal L(L^p(D))}
   \,\|(-A_i)^{\frac{\varepsilon}{2}}F_{i,n}(s)\|_{L^p(D)}\,ds \\
&\le C_{\varepsilon}\int_0^t (t-s)^{-(1-\frac{\varepsilon}{2})}\,
   \|(-A_i)^{\frac{\varepsilon}{2}}F_{i,n}(s)\|_{L^p(D)}\,ds\\
&\leq C_{\varepsilon}\int_0^t (t-s)^{-\frac{2-\varepsilon}{2}}||F_{i,n}(s)||_{W^{\varepsilon,p}(D)}ds\\
&\leq C_{\varepsilon,t} \sup_{s\leq t} ||F_{i,n}(s)||_{W^{\varepsilon,p}(D)}ds < \infty.
\end{align*}
Therefore, $v_{i,n}(t) \in W^{2,p}(D)$ for any $t\in (0,T]$. Moreover, since $v_{i,n}$ satisfies
\[
\partial_t v_{i,n}(t,x)= d_i\Delta v_{i,n}(t,x) + f_{i,n}(u_n(t,x))
\]
weakly $(t,x)$ and $\Delta v_{i,n}\in L^p((0,T]\times D)$ and $f_{i,n}(u_n)\in L^p( (0,T]\times D)$, we deduce that $\partial_t v_{i,n} \in L^p((0,T)\times D)$. Thus, $v_{i,n}$ as defined in \eqref{eq: v_i,n definition} is a regular function.
\section{Main results}\label{section: Main results}
Consider the initial $m\times r$ system \eqref{eq: initial mxr system} along with the truncated global mild solution \eqref{eq: truncated global solution mxr}. As we mentioned in the Introduction, our goal is to prove the uniform in $n$ estimate \eqref{eq: uniform bound}. In this section, we first prove global existence for the $2\times r$ system. Then, we prove global existence for the $m\times r$ system by induction. We record here two elementary properties of the heat kernel $G(t,x,y)$. The first one is
\begin{equation}\label{eq: int of heat kernel <= 1}
    0\leq \int_D G(t,x,y) dy \leq 1.
\end{equation}
The second one is a direct consequence of H\"older's inequality conjoined with Assumption \ref{assump: heat kernel singularity} and reads as
\begin{equation}\label{eq: Heat kernel bound G^p/p-1}
\int_D G^{\frac{p}{p-1}}(t,x,y)dy \leq ||G||_{L^{\infty}}^{\frac{1}{p-1}}||G||_{L^1}\leq Ct^{-\frac{d}{2(p-1)}}.
\end{equation}
\subsection{Global existence for the $2\times r$ system}
Let $u = (u_1,u_2)$ and consider the system
\begin{equation}\label{eq: 2xr system}
\begin{cases}
    \partial_t u_1(t,x) - d_1\Delta u_1(t,x) = f_1(u(t,x)) + \sum\limits_{k=1}^r \sigma_{1k}(u(t,x))\dot{W}_k(t,x),\\
    \partial_t u_2(t,x) - d_2\Delta u_2(t,x) = f_2(u(t,x)) + \sum\limits_{k=1}^r \sigma_{2k}(u(t,x))\dot{W}_k(t,x),\\
    u_i = 0 \text{ or } \dfrac{\partial u_i}{\partial n} = 0 \text{ on } (0,T)\times \partial D, i\in \{1,2\},\\
    u_i(0,x) = u_{i0}(x) \text{ in }  D, i\in \{1,2\},
\end{cases}
\end{equation}
along with the truncated global solution \eqref{eq: truncated global solution mxr}. We write $u_n(t,x) = (u_{1,n}(t,x),u_{2,n}(t,x))$. As a preliminary step, we show that we can obtain uniform-in-$n$ bounds on the $L^p(Q_T)$ norms of $u_{i,n}$ for a sufficiently small time $T$. Then, we will use Lemma \ref{lemma: lift L^p integral to L^infinity} to lift the $L^p$ norms into $L^{\infty}$.
\begin{lemma}\label{lemma: For every p there exists a T_p, C_p such that bound for small times}
For any $p\geq 1$ there exists a time $T_p$ and a constant $C_p$, both of which depend on $p$, such that
\[
\mathbb E||u_{i,n}||^p_{L^p(Q_{T_p})}\leq C_{p}\left(\sum_{i=1}^2 ||u_{i0}||^p_{L^{\infty}(D)}+1\right),
\]
for any $i=1,2$, uniformly in $n\in \mathbb N$.
\end{lemma}
\begin{proof}
    In view of our Introduction, system \eqref{eq: 2xr system} admits a \textit{truncated} global mild solution $(u_{1,n}(t,x),u_{2,n}(t,x))$ given by
\begin{align*}
\begin{split}
    u_{i,n}(t,x) = \int_D G_i(t,x,y)u_0(y)dy &+ \int_0^t\int_D G_i(t-s,x,y)f_{i,n}(u_n(s,y))dyds\\
    &+\sum_{k=1}^r\int_0^t\int_D G_i(t-s,x,y)\sigma_{ik,n}(u_n(s,y))W_k(dyds).
\end{split}
\end{align*}
Set
\begin{equation}\label{eq: Y_n(t) definition}
Y_n(t) := ||u_{1,n}||_{L^p(Q_t)}^p+||u_{2,n}||_{L^p(Q_t)}^p = \sum_{i=1}^2 ||u_{i,n}||_{L^p(Q_t)}^p.\end{equation}
We wish to prove that there exists a time $T_p$ such that for all $n\in \mathbb N$, $\mathbb EY_n(T_p)\leq C$, for a constant $C$ that does not depend on $n$.

Let $Z_{i,n}$ be the stochastic convolution
\begin{equation}\label{eq: truncated stochastic conv}
Z_{i,n}(t,x) = \sum_{k=1}^r \int_0^t\int_{D} G_{i}(t-s,x,y)\sigma_{ik,n}(u_n(s,y))W_k(dyds),
\end{equation}
and set 
\[
v_{i,n}(t,x) := u_{i,n}(t,x) - Z_{i,n}(t,x).
\]
Then, $v_{i,n}$ is given by the integral equation
\begin{equation}\label{eq: v_{i,n} mild form}
v_{i,n}(t,x) = \int_D G_{i}(t,x,y)u_{i0}(y)dy + \int_0^t\int_{D} G_{i}(t-s,x,y)f_{i,n}(u_n(s,y))dyds,
\end{equation}
meaning that $v_{1,n}(t,x)$ and $v_{2,n}(t,x)$ solve the system of partial differential equations
\begin{equation}\label{eq: 2x2 system v_i,n}
\begin{cases}
    \partial_t v_{1,n}(t,x)- d_1\Delta v_{1,n}(t,x) = f_{1,n}(u_n(t,x)),\\
    \partial_t v_{2,n}(t,x) - d_2\Delta v_{2,n}(t,x) = f_{2,n}(u_n(t,x)),\\
    v_{i,n}(0,x) = u_{i0}(x).
\end{cases}
\end{equation}
Assumption \ref{assump: triangular mass control f} (triangular mass control) in this two dimensional case reduces to
\begin{align}
    &f_{1,n}(a_1,a_2) \leq C_1(1+a_1+a_2), \label{eq: f_1(u) 2xr}\\
    &f_{1,n}(a_1,a_2)+f_{2,n}(a_1,a_2) \leq C_2(1+a_1+a_2),\label{eq: f_1(u)+f_2(u) 2xr}
\end{align}
for any $a_1,a_2\geq 0$. By Corollary \ref{corollary: nonnegativity of u_i,n}, $u_{i,n}(t,x)$ are nonnegative a.e. Therefore, applying \eqref{eq: f_1(u) 2xr} to the first equation of \eqref{eq: 2x2 system v_i,n} gives
\[
\partial_t v_{1,n}(t,x)- d_1\Delta v_{1,n}(t,x) \leq C_1(1+u_{1,n}(t,x)+u_{2,n}(t,x)).
\]
Now we are allowed to apply Lemma \ref{lemma: Duality lemma} with the choices $v_1=v_{1,n}(t,x),\theta_i = 0,v_2=0,$ and $H(t,x)=|C_1|(1+u_{1,n}(t,x)+u_{2,n}(t,x)).$
Indeed, notice that $H$ is nonnegative by Corollary \ref{corollary: nonnegativity of u_i,n}. It is also in $L^p(Q_T)$ by the bound \eqref{eq: pth moment bound}, since the $L^p(Q_T)$ norm is always controlled above by the $L^\infty(Q_T)$ norm on bounded domains. Therefore,
\begin{align*}
||v_{1,n}^{+}||_{L^p(Q_t)} &\leq C_1\left(||v_{1,n}(0)||_{L^p(D)}+ \int_0^t ||H(s)||_{L^p(D)}ds\right)\\
&\leq C_1\left(||v_{1,n}(0)||_{L^p(D)}+ \int_0^t (1+||u_{1,n}(s)||_{L^p(D)}+||u_{2,n}(s)||_{L^p(D)})ds\right),
\end{align*}
for a constant $C_1$ that depends only on $p,d_1$, and $d_2$. Our assumptions guarantee (see Lemma \ref{lemma: non negativity of solutions}) that $u_{i,n}$ is nonnegative, but not that $v_{i,n}$ is nonnegative. However, the nonnegativity of $u_{i,n}$ enables us to bound the negative part of $v_{i,n}$ above by the stochastic convolution,
\begin{equation}\label{eq: how to bound negative part}
||v_{1,n}^{-}||_{L^p(Q_t)} = ||(u_{1,n}-Z_{1,n})^{-
}||_{L^p(Q_t)} =||(Z_{1,n}-u_{1,n})^{+}||_{L^p(Q_t)}\leq ||Z_{1,n}^{+}||_{L^p(Q_t)}.
\end{equation}
Therefore, we deduce that
\begin{equation}\label{eq: bound negative part 2xr}
    ||v_{1,n}^{-}||_{L^p(Q_t)}\leq||Z_{1,n}||_{L^p(Q_t)}.
\end{equation}
And by the triangle inequality
\begin{equation}\label{eq: ||v_1n|| 2xr}
||v_{1,n}||_{L^p(Q_t)}\leq C_{1}\Bigg(||v_{1,n}(0)||_{L^p(D)}+||Z_{1,n}||_{L^p(Q_t)}+t+\int_0^t (||u_{1,n}(s)||_{L^p(D)}+||u_{2,n}(s)||_{L^p(D)})ds\Bigg).
\end{equation}
Similarly, adding the two equations of \eqref{eq: 2x2 system v_i,n} together and applying \eqref{eq: f_1(u)+f_2(u) 2xr} gives
\[
\partial_t v_{2,n}(t,x) - d_2\Delta v_{2,n}(t,x) \leq \partial_t (-v_{1,n}(t,x))-d_1\Delta (-v_{1,n}(t,x)) + C_2(1+u_{1,n}(t,x)+u_{2,n}(t,x)).
\]
Thus, by Lemma \ref{lemma: Duality lemma} with $H(t,x)=|C_2|(1+u_{1,n}(t,x)+u_{2,n}(t,x))$
\begin{align*}
||v_{2,n}^{+}||_{L^p(Q_t)} &\leq C_{2}\Bigg( ||v_{1,n}(0)||_{L^p(D)}+||v_{2,n}(0)||_{L^p(D)}+ ||v_{1,n}||_{L^p(Q_t)}\\
&\qquad\qquad + \int_0^t (1+||u_{1,n}(s)||_{L^p(D)}+||u_{2,n}(s)||_{L^p(D)})ds\Bigg).
\end{align*}
The negative part of $v_{2,n}(t,x)$ is similarly bounded above by the stochastic convolution,
\[
||v_{2,n}^{-}||_{L^p(Q_t)} \leq ||Z_{2,n}||_{L^p(Q_t)}.
\]
Therefore, by the triangle inequality,
\begin{equation}\label{eq: v_2,n bound}
\begin{aligned}
||v_{2,n}||_{L^p(Q_t)}\leq C_{2}\Bigg(\sum_{i=1}^2 &||v_{i,n}(0)||_{L^p(D)}+||v_{1,n}||_{L^p(Q_t)}+||Z_{2,n}||_{L^p(Q_t)}+t\\
&+\int_0^t (||u_{1,n}(s)||_{L^p(D)}+||u_{2,n}(s)||_{L^p(D)})ds\Bigg),
\end{aligned}
\end{equation}
and using the estimate \eqref{eq: ||v_1n|| 2xr} for $||v_{1,n}||_{L^p(Q_t)}$,
\begin{equation}\label{eq: ||v_2,n|| 2xr}
\begin{aligned}
||v_{2,n}||_{L^p(Q_t)}\leq C_2\Bigg( \sum_{i=1}^2 &||v_{i,n}(0)||_{L^p(D)}+\sum_{i=1}^2 ||Z_{i,n}||_{L^p(Q_t)}+t\\
&+\int_0^t (||u_{1,n}(s)||_{L^p(D)}+||u_{2,n}(s)||_{L^p(D)})ds\Bigg).
\end{aligned}
\end{equation}
Applying H\"older's inequality to the integral and taking the $p$-th power gives
\begin{align}
||v_{1,n}||^p_{L^p(Q_t)} &\leq C_{p}\Bigg( ||v_{1,n}(0)||^p_{L^p(D)}+||Z_{1,n}||^p_{L^p(Q_t)}+t^p\notag\\
&\qquad\qquad+t^{p-1}\int_0^t (||u_{1,n}(s)||^p_{L^p(D)}+||u_{2,n}(s)||^p_{L^p(D)})ds\Bigg)\notag\\
&=C_{p}\Big( ||v_{1,n}(0)||^p_{L^p(D)}+||Z_{1,n}||^p_{L^p(Q_t)}+t^p\notag\\
&\qquad\qquad+t^{p-1}(||u_{1,n}||_{L^p(Q_t)}^p+||u_{2,n}||_{L^p(Q_t)}^p)\Big)\label{eq: b||v 1n||},
\end{align}
and
\begin{align}
||v_{2,n}||^p_{L^p(Q_t)} &\leq C_{p,T}\Bigg(\sum_{i=1}^2 ||v_{i,n}(0)||^p_{L^p(D)}+\sum_{i=1}^2||Z_{i,n}||_{L^p(Q_t)}^p+t^p\notag\\
&\qquad\qquad+t^{p-1}\int_0^t (||u_{1,n}(s)||^p_{L^p(D)}+||u_{2,n}(s)||^p_{L^p(D)})ds\Bigg)\notag\\
&= C_{p,T}\Bigg(\sum_{i=1}^2 ||v_{i,n}(0)||^p_{L^p(D)}+\sum_{i=1}^2||Z_{i,n}||_{L^p(Q_t)}^p+t^p\notag\\
&\qquad\qquad+t^{p-1}(||u_{1,n}||_{L^p(Q_t)}^p+||u_{2,n}||_{L^p(Q_t)}^p)\Bigg).\label{eq: b||v_2n||}
\end{align}
But now, recall that
\[
u_{i,n}(t,x) = v_{i,n}(t,x) + Z_{i,n}(t,x),
\]
from which it follows that
\begin{equation}\label{eq: b||u_in||}
||u_{i,n}||_{L^p(Q_t)}^p\leq C_p\left(||v_{i,n}||^p_{L^p(Q_t)}+||Z_{i,n}||^p_{L^p(Q_t)}\right).
\end{equation}
If we set 
\begin{equation}\label{eq: Y 2xr}
Y_n(t) := ||u_{1,n}||_{L^p(Q_t)}^p+||u_{2,n}||_{L^p(Q_t)}^p = \sum_{i=1}^2 ||u_{i,n}||_{L^p(Q_t)}^p,
\end{equation}
and insert \eqref{eq: b||v 1n||}, \eqref{eq: b||v_2n||}, and \eqref{eq: b||u_in||} into \eqref{eq: Y 2xr},  we get
\begin{equation}\label{eq: Ineq for Y(t)}
Y_n(t) \leq C_{p}\left(\sum_{i=1}^2 ||v_{i,n}(0)||_{L^p(D)}^p+\sum_{i=1}^2 ||Z_{i,n}||_{L^p(Q_t)}^p+t^p+t^{p-1}Y_n(t)\right).
\end{equation}
Then, taking expectations and applying \eqref{eq: Lp norm Z_i} gives
\[
\mathbb E Y_n(t)
\le
C_{p}\left(
\mathbb E \sum_{i=1}^2 ||v_{i,n}(0)||^p_{L^p(D)}+t^{a+1}\bigl(t+\mathbb E Y_n(t)\bigr)+t^p+t^{p-1}\mathbb E Y_n(t)
\right).
\]
By a mere rearrangement of terms, we deduce the inequality
\begin{equation}\label{eq: E Y_n(t) bound ineq}
    (1-C_{p}(t^{a+1}+t^{p-1}))\mathbb EY_n(t)\leq C_{p}\Bigg(\mathbb E \sum_{i=1}^2 ||v_{i,n}(0)||_{L^p(D)}^p+t^{a+2}+t^p\Bigg).
\end{equation}
Since $a+1>0$ and $p-1>0$, the term $
C_{p}(t^{a+1}+t^{p-1})$ decreases to $0$
as $t\rightarrow 0$. Therefore, there exists a time $T_p>0$ such that $C_{p}(t^{a+1}+t^{p-1}) < \frac12$ for all $t\in [0,T_p]$. Moreover, the time $T_p$ depends on $p$ but does not depend on the initial data, since the constant $C_{p}$ does not itself depend on the initial data. Therefore, we may solve \eqref{eq: E Y_n(t) bound ineq} on $[0,T_p]$ and recall that $v_{i,n}(0,\cdot) = u_{i0}(\cdot)$ to get
\begin{equation}\label{eq: EY_n(T_0) upper bound}
\mathbb E Y_n(T_p) \leq C_{p,T_p}\Bigg(\sum\limits_{i=1}^2 ||u_{i0}||_{L^{\infty}(D)}^p+1\Bigg).
\end{equation}
Therefore, on $[0,T_p]$ we obtain the bounds
\begin{equation*}
    \mathbb E ||u_{i,n}||_{L^p(Q_{T_p)}}^p \leq C_{p,T_p}\Bigg(\sum\limits_{i=1}^2 ||u_{i0}||_{L^p(D)}^p+1\Bigg),
\end{equation*}
for any $i=1,2$, uniformly in $n\in \mathbb N$.
\end{proof}
Now we show that for any $p$ sufficiently large, we can lift the Lebesgue integral of $v_{i,n}$ from $L^p(Q_T)$ to $L^{\infty}(Q_T)$.
\begin{lemma}\label{lemma: lift L^p integral to L^infinity}
   Let $p>1+d/2$ and $T>0$. If $f_i\in L^p(Q_T)$ then 
   \[
   \int_0^t\int_D G_i(t-s,x,y)f_i(s,y)dyds\in L^{\infty}(Q_T),
   \]
and in particular
   \[
\left\|
\int_0^\cdot \int_D G_i(\,\cdot - s,\bullet,y)\, f_i(s,y)\, dy\, ds
\right\|_{L^\infty(Q_T)}
\leq
C_{p,i}\, T^{\frac{p-1-d/2}{p}}
\|f_i\|_{L^p(Q_T)}.
\]
\end{lemma}
\begin{proof}
We apply H\"older's inequality and \eqref{eq: Heat kernel bound G^p/p-1} to get
\begin{align*}
\left|\int_0^t\int_D G_i(t-s,x,y)f_{i}(s,y)dyds \right| &\leq \left(\int_0^t\int_D G_i^{\frac{p}{p-1}}(t-s,x,y)dyds\right)^{\frac{p-1}{p}}\\
&\times\left(\int_0^t\int_D |f_{i}(s,y)|^pdyds \right)^{\frac{1}{p}}\\
&\leq C_{p,i}\Big(\int_0^t (t-s)^{-\frac{d}{2(p-1)}}ds\Big)^{\frac{p-1}{p}}||f_{i}||_{L^p(Q_T)}\\
& \leq  C_{p,i}T^{\frac{p-1}{p}-\frac{d}{2p}}||f_{i}||_{L^p(Q_{t})},
\end{align*}
where $p$ has been chosen larger than $1+d/2$.
\end{proof}
Lemma \ref{lemma: For every p there exists a T_p, C_p such that bound for small times} allows us to bound the $L^p$ norm uniformly in $n$ for small times $T$, and Lemma \ref{lemma: lift L^p integral to L^infinity} upgrades the $L^p$ bounds to $L^{\infty}$ bounds. We now combine these two results to prove our first main result.
\begin{theorem}\label{thm: Theorem 2xr}
    Under Assumptions \ref{assump: initial data are >= 0 and in L^oo}-\ref{assump: Integrability of kernel L}, for any fixed time horizon $T>0$,
    \begin{equation*}
\sup_{n\in \mathbb N} \mathbb E \sup_{t\in [0,T]}\sup_{x\in D}\sup_{i=1,2} |u_{i,n}(t,x)|^p < \infty.
\end{equation*}
In particular, system \eqref{eq: 2xr system} admits a unique global mild solution.
\end{theorem}
\begin{proof}
Let $\mu>0$ be given by the polynomial growth of $f_i$ (Assumption \ref{assump: Polynomial growth f}). Lemma \ref{lemma: For every p there exists a T_p, C_p such that bound for small times} showed us that for any $\mu p\geq 1$, there exists a sufficiently small time $T_{\mu p}$ such that
\[
\mathbb E||u_{i,n}||^{\mu p}_{L^{\mu p}(Q_{T_{\mu p}})}\leq C_{p}\left(\sum_{i=1}^2 ||u_{i0}||^p_{L^{\infty}(D)}+1\right).
\]
Now we lift this $L^p$ bound into an $L^{\infty}$ bound. Recall by \eqref{eq: v_{i,n} mild form} that
\[
v_{i,n}(t,x) = \int_D G_i(t,x,y)u_{i0}(y)dy + \int_0^t\int_D G_i(t-s,x,y)f_{i,n}(u_n(s,y))dyds.
\]
By \eqref{eq: int of heat kernel <= 1}, the first term is bounded as
\begin{equation*}
    \Bigg |\int_D G_i(t,x,y)u_{i0}(y)dy \Bigg | \leq ||u_{i0}||_{L^\infty(D)}\int_D G_i(t,x,y)dy \leq ||u_{i0}||_{L^\infty(D)}.
\end{equation*}
 Therefore, we apply Lemma \ref{lemma: lift L^p integral to L^infinity} to get, for $\mu p>1+d/2$,
\begin{align*}
\sup_{t\in [0,T_{\mu p}]}\sup_{x\in D}\sup_{i=1,2} |v_{i,n}(t,x)| \leq C_{p,T_{\mu p}}\sup_{i=1,2} \Big(||u_{i0}||_{L^{\infty}(D)} + ||f_{i,n}(u_n)||_{L^p(Q_{T_{\mu p}})}\Big).
\end{align*}
By Assumption \ref{assump: Polynomial growth f} (polynomial growth of $f$), we get
\[
||f_{i,n}(u_n)||_{L^p(Q_{T_{\mu p}})}\leq C\Big(1 + ||u_{1,n}||^{\mu}_{L^{\mu p}(Q_{T_{\mu p}})}+||u_{2,n}||^{\mu}_{L^{\mu p}(Q_{T_{\mu p}})}\Big),
\]
and so by Lemma \ref{lemma: For every p there exists a T_p, C_p such that bound for small times}
\[
\mathbb E \sup_{t\in [0,T_{\mu p}]}\sup_{x\in D}\sup_{i=1,2}|v_{i,n}(t,x)|^p\leq C_{p,\mu, T_{\mu p}}\Bigg(1+\sum\limits_{i=1}^2 ||u_{i0}||_{L^{\infty}(D)}^{\mu p}+ \sup_{i=1,2}||u_{i0}||^p_{L^{\infty}(D)}\Bigg).
\]
For $p$ sufficiently large, we may rewrite this bound by absorbing constants into $C_{p,T_{\mu p}}$ as
\[
\mathbb E \sup_{t\in [0,T_{\mu p}]}\sup_{x\in D}\sup_{i=1,2}|v_{i,n}(t,x)|^p\leq C_{p,\mu,T_{\mu p}}\Bigg(1+\sup_{i=1,2} ||u_{i0}||^p_{L^{\infty}(D)}+\sup_{i=1,2}||u_{i0}||^{\mu p}_{L^{\infty}(D)}\Bigg).
\]
As for the stochastic convolution, we have by \eqref{eq: sup norm Z_i} that
\begin{align*}
\mathbb E \sup_{t\in [0,T_{\mu p}]}\sup_{x\in D} \sup_{i=1,2}|Z_{i,n}(t,x)|^p &\leq C_{p,T_{\mu p}}(1+\mathbb EY_n(T_{\mu p}))\leq C_{p,T_{\mu p}}\Big(1+\sup_{i=1,2}||u_{i0}||_{L^p(D)}^p\Big),
\end{align*}
where $Y_n(t)$ was defined in \eqref{eq: Y 2xr} and we used the uniform estimate \eqref{eq: EY_n(T_0) upper bound}. So, all in all, by invoking the definition of $u_{i,n}(t,x)$
\[
u_{i,n}(t,x) = v_{i,n}(t,x) + Z_{i,n}(t,x),
\]
we conclude that there exists a constant $C_{p,\mu,T_{\mu p}}>0$ such that
\begin{equation}\label{eq: sup bound on [0,T0]}
\mathbb E\sup_{t\in [0,T_{\mu p}]}\sup_{x\in D}\sup_{i=1,2} |u_{i,n}(t,x)|^p\leq C_{p,\mu,T_{\mu p}}\Bigg(1+\sup_{i=1,2} ||u_{i0}||^p_{L^{\infty}(D)}+\sup_{i=1,2}||u_{i0}||^{\mu p}_{L^{\infty}(D)}\Bigg),
\end{equation}
uniformly in $n\in \mathbb N$ for some $p$ sufficiently large.\\[1em]
\noindent 
We now extend this small-time estimate to intervals $[T_{\mu p}, 2T_{\mu p}]$,...,$[(N-1)T_{\mu p},NT_{\mu p}]$, where $N\in \mathbb N$. Since $\{u_{n}(t,\cdot)\}_{t\in [0,T]}$ is a Markov process (see, for instance, section 9.2 in \cite{da2014stochastic}, or \cite{cerrai2003stochastic} for an application of the Markov property that is similar to our case), 
we may repeat the above argument on the time interval $[jT_{\mu p},(j+1)T_{\mu p}]$. Then, due to the Markov property,
\begin{equation*}
\mathbb{E}\sup_{t\in [jT_{\mu p},(j+1)T_{\mu p}]}\sup_{x\in D}\sup_{i=1,2}
|u_{i,n}(t,x)|^p
\le
C_{p,\mu,T_{\mu p}}\Big(\mathbb{E}\sup_{i=1,2}\|u_{i,n}(t_j)\|_{L^\infty(D)}^p+1\Big),
\end{equation*}
where $t_j := jT_{\mu p}$. Using $||u_{i,n}(t_j)||^p_{L^{\infty}(D)}\leq \sup\limits_{t\in [t_{j-1},t_{j}]} ||u_n(t)||_{L^{\infty}(D)}^p$ and \eqref{eq: sup bound on [0,T0]} yields
\begin{align*}
\mathbb{E}\sup_{t\in [jT_{\mu p},(j+1)T_{\mu p}]}\sup_{x\in D}\sup_{i=1,2}
|u_{i,n}(t,x)|^p
&\le
C_{p,\mu,T_{\mu p}}\Big(\mathbb{E}\sup_{i=1,2}\|u_{i,n}(t_j)\|_{L^\infty(D)}^p+1\Big)\\
&\leq C_{p,\mu,T_{\mu p}}\Big(C_{p,\mu,T_{\mu p}}(\mathbb E\sup_{i=1,2} ||u_{i,n}(t_{j-1})||^p_{L^{\infty}(D)}+1)+1\Big)\\
&\vdots\\
&\leq C_{p,\mu,T_{\mu p}}^j\left(\sup_{i=1,2}||u_{i0}||_{L^{\infty}(D)}^p+1\right)+ C_{p,\mu,T_{\mu p}}^{j-1} + ... + C_{p,\mu,T_{\mu p}}.
\end{align*}
And since
\[
\sup_{t\in [0,NT_{\mu p}]}\sup_{x\in D}\sup_{i=1,2} |u_{i,n}(t,x)|^p =\max\limits_{j=1,...,N-1}\sup_{t\in [jT_{\mu p},(j+1)T_{\mu p}]}\sup_{x\in D}\sup_{i=1,2} |u_{i,n}(t,x)|^p,
\]
we finally get, for any $N\in \mathbb N$,
\[\mathbb{E}\sup_{t\in [0,ΝT_{\mu p}]}\sup_{x\in D}\sup_{i=1,2}
|u_{i,n}(t,x)|^p
\leq C_{p,\mu,T_{\mu p}}\left(\sup_{i=1,2}||u_{i0}||^p_{L^{\infty}(D)}+1\right),\]
where the constant $C_{p,\mu,T_{\mu p}}$ depends on $p,\mu$, and $T_{\mu p}$, but not on $n$. The proof is completed.
\end{proof}
\subsection{Global existence for the $m\times r$ system}
Now consider the $m\times r$ system:
\begin{equation}\label{eq: Theorem mxr system}
    \begin{cases}
        \partial_t u_i(t,x) - d_i\Delta u_i(t,x) = f_i(u(t,x))+\sum\limits_{k=1}^r \sigma_{ik}(u(t,x))\dot{W}_k(t,x)\text{ in } (0,T)\times D,\\
        u_i = 0 \text{ or } \dfrac{\partial u_i}{\partial n} = 0 \text{ on } (0,T)\times \partial D,\\
        u_i(0,x) = u_{i0}(x).
    \end{cases}
\end{equation}
We write $u_n(t,x) = (u_{1,n}(t,x),...,u_{m,n}(t,x))$ for the \textit{truncated} global mild solution, as defined in \eqref{eq: truncated global solution mxr}.
\begin{theorem}\label{thm: Theorem mxr}
    Under Assumptions \ref{assump: initial data are >= 0 and in L^oo}-\ref{assump: Integrability of kernel L}, for any fixed time horizon $T>0$,
    \begin{equation*}
\sup_{n\in \mathbb N} \mathbb E \sup_{t\in [0,T]}\sup_{x\in D}\sup_{i=1,...,m} |u_{i,n}(t,x)|^p < \infty.
\end{equation*}
In particular, system \eqref{eq: Theorem mxr system} has a unique global mild solution.
\end{theorem}
\begin{proof}
Fix $p>1+d/2$. Set
    $$U(t,x) :=u_{1,n}(t,x)+...+u_{m,n}(t,x).$$
We consider the functions  
    $$v_{i,n}(t,x) = u_{i,n}(t,x) - Z_{i,n}(t,x)$$
    that solve the system of partial differential equations
\begin{align}\label{truncated v_i PDE system mxm}
\begin{split}
&\partial_t v_{1,n}(t,x) - d_1\Delta v_{1,n}(t,x) = f_{1,n}(u_n(t,x))\\
&\qquad\qquad\qquad\vdots\\
&\partial_t v_{m,n}(t,x) - d_m\Delta v_{m,n}(t,x) = f_{m,n}(u_n(t,x)).
\end{split}
\end{align}
We will prove a bound on $||v_{k,n}||_{L^p(Q_t)}$ by induction.
Fix an $i=1,...,m$ and consider the first $i$ equations of \eqref{truncated v_i PDE system mxm}
 \begin{equation}\label{eq: i eq}
    \begin{cases}
        (\partial_t - d_1\Delta)v_{1,n}(t,x) = f_{1,n}(u_n(t,x)),\\
        (\partial_t - d_2\Delta)v_{2,n}(t,x) = f_{2,n}(u_n(t,x)),\\
       \qquad\qquad\qquad \vdots\\
        (\partial_t - d_i\Delta)v_{i,n}(t,x) = f_{i,n}(u_n(t,x)).
    \end{cases}
    \end{equation}
Using the mass control of the first $i$ reaction terms $f_{1,n}(u_n(t,x))+...+f_{i,n}(u_n(t,x))$ (Assumption \ref{assump: triangular mass control f})
\[
f_{1,n}(u_n(t,x))+...+f_{i,n}(u_n(t,x)) \leq C(1+u_{1,n}(t,x) + ... + u_{i,n}(t,x)),
\]
and adding the equations in \eqref{eq: i eq} we obtain
    \[
    \sum_{j=1}^{i} (\partial_t -d_{j}\Delta)v_{j,n}(t,x) \leq C(1+U(t,x)), \hspace{0.1cm} \text{ for any } i\in \{1,...,m\},
    \]
    where $C=\max \{C_1,...,C_m\}$. Equivalently,
    \[
    (\partial_t - d_i\Delta )v_{i,n}(t,x) \leq -\sum_{j=1}^{i-1} (\partial_t - d_j\Delta)v_{j,n}(t,x) + C(1+U(t,x)). 
    \]
    Now set $H(t,x) = |C|(1+U(t,x))$. By Corollary \ref{corollary: nonnegativity of u_i,n}, $H(t,x)\geq 0$ and is also in $L^p(Q_t)$ by the bound \eqref{eq: pth moment bound}. Then, by Lemma \ref{lemma: Duality lemma}
    \[
    ||v_{i,n}^{+}||_{L^p(Q_t)}\leq C\left(\sum_{j=1}^{i} ||v_{j,n}(0)||_{L^p(D)}+\sum_{j=1}^{i-1} ||v_{j,n}||_{L^p(Q_t)} + \int_0^t ||H(s)||_{L^p(D)}ds\right).
    \]
   From the argument that gave us \eqref{eq: bound negative part 2xr}, we can control the negative part above by the stochastic convolution as
    \[
    ||v_{i,n}^{-}||_{L^p(Q_t)}\leq ||Z_{i,n}||.
    \]
    Therefore, for every $i=1,...,m$
    \begin{equation}\label{v_i before induction mxm system}
        ||v_{i,n}||_{L^p(Q_t)}\leq C\left(\sum_{j=1}^{i} ||v_{j,n}(0)||_{L^p(D)}+||Z_{i,n}||_{L^p(Q_t)}+\sum_{j=1}^{i-1}||v_{j,n}||_{L^p(Q_t)} + \int_0^t ||H(s)||_{L^p(D)}ds\right)
    \end{equation}
    We will prove by induction on $i$ that
    \begin{equation}\label{v_i after induction mxm system}
        ||v_{k,n}||_{L^p(Q_t)}\leq C\left(\sum_{j=1}^{k} ||v_{j,n}(0)||_{L^p(D)}+\sum_{j=1}^{k} ||Z_{j,n}||_{L^p(Q_t)}+\int_0^t ||H(s)||_{L^p(D)}ds\right).
    \end{equation}
    We have already shown in the proof of Theorem \ref{thm: Theorem 2xr} that \eqref{v_i after induction mxm system} holds for $i=1$. Indeed,  \eqref{eq: ||v_1n|| 2xr} reads as
    \[
    ||v_{1,n}||_{L^p(Q_t)} \leq C\left(||v_{1,n}(0)||_{L^p(D)}+||Z_{1,n}||_{L^p(Q_t)}+\int_0^t ||H(s)||_{L^p(D)}ds\right).
    \]
    Now, for the inductive step, assume that for any $k\leq i$,
    \[
    ||v_{k,n}||\leq C\left(\sum_{j=1}^{k} ||v_{j,n}(0)||_{L^p(D)}+\sum_{j=1}^{k} ||Z_{j,n}||_{L^p(Q_t)}+\int_0^t ||H(s)||_{L^p(D)}ds\right).
    \]
    Then, by \eqref{v_i before induction mxm system},
    \begin{align*}
        ||v_{i+1}||_{L^p(Q_t)}&\leq C\left(\sum_{j=1}^{i+1} ||v_{j,n}(0)||_{L^p(D)}+||Z_{i+1}||_{L^p(Q_t)} + \sum_{j=1}^i ||v_{j,n}||_{L^p(Q_t)} + \int_0^t ||H(s)||_{L^p(D)}ds\right)\\
        &\leq C\left(\sum_{j=1}^{i+1} ||v_{j,n}(0)||_{L^p(D)}+||Z_{i+1}||_{L^p(Q_t)}+\sum_{j=1}^i ||Z_{j,n}||_{L^p(Q_t)}+\int_0^t ||H(s)||_{L^p(D)}ds\right)\\
        &= C\left(\sum_{j=1}^{i+1} ||v_{j,n}(0)||_{L^p(D)}+\sum_{j=1}^{i+1}||Z_{j}||_{L^p(Q_t)}+\int_0^t ||H(s)||_{L^p(D)}ds\right),
    \end{align*}
    from which we obtain \eqref{v_i after induction mxm system}. So now for all $1\leq k\leq m$ and $t\in [0,T]$, using H\"older's inequality
    \begin{align*}
        ||v_{k,n}||_{L^p(Q_t)}&\leq C\left(\sum_{j=1}^{k} ||v_{j,n}(0)||_{L^p(D)}+\sum_{j=1}^k ||Z_{j,n}||_{L^p(Q_t)}+\int_0^t ||H(s)||_{L^p(D)}ds\right)\\
        &\leq C\left(\sum_{j=1}^{k} ||v_{j,n}(0)||_{L^p(D)}+\sum_{j=1}^k ||Z_{j,n}||_{L^p(Q_t)}+t+t^{\frac{p-1}{p}}\left(\int_0^t ||U(s)||^p_{L^p(D)}ds\right)^{\frac{1}{p}}\right)\\
        &\leq C\left(\sum_{j=1}^{k} ||v_{j,n}(0)||_{L^p(D)}+\sum_{j=1}^k ||Z_{j,n}||_{L^p(Q_t)}+t+t^{\frac{p-1}{p}} \left(\int_0^t \sum_{j=1}^k ||u_{j,n}(s)||^p_{L^p(D)}ds\right)^{\frac{1}{p}}\right)
    \end{align*}
    Taking the $p$-th power
    \[
    ||v_{k,n}||^p_{L^p(Q_t)}\leq C_{p}\left(\sum_{j=1}^{k} ||v_{j,n}(0)||^p_{L^p(D)}+\sum_{j=1}^k ||Z_{j,n}||^p_{L^p(Q_t)} +t^p+ t^{p-1}\int_0^t \sum_{j=1}^k ||u_{j,n}(s)||^p_{L^p(D)}ds \right).
    \]
    And since $||u_{k,n}||^p_{L^p(Q_t)} \leq C_p(||v_{k,n}||_{L^p(Q_t)}^p + ||Z_{k,n}||_{L^p(Q_t)}^p)$,
    \[
    ||u_{k,n}||^p_{L^p(Q_t)} \leq C_{p}\left(\sum_{j=1}^{k} ||v_{j,n}(0)||^p_{L^p(D)}+\sum_{j=1}^k ||Z_{j,n}||^p_{L^p(Q_t)} + t^p+ t^{p-1}\int_0^t \sum_{j=1}^k ||u_{j,n}(s)||^p_{L^p(D)}ds \right).
    \]
 Equivalently,
 \[
||u_{k,n}||^p_{L^p(Q_t)} \leq C_{p}\left(\sum_{j=1}^{k} ||v_{j,n}(0)||^p_{L^p(D)}+\sum_{j=1}^k ||Z_{j,n}||^p_{L^p(Q_t)} + t^p+t^{p-1}\sum_{j=1}^k ||u_{j,n}||^p_{L^p(Q_t)}ds \right).
 \]
 Now, as in the proof of Theorem \ref{thm: Theorem 2xr}, we set
   \[
   Y_n(t) := \sum_{j=1}^m ||u_{j,n}||_{L^p(Q_t)}^p.
   \]
   Then,
   \[Y_n(t) \leq C_{p}\left(\sum_{j=1}^{m} ||v_{j,n}(0)||^p_{L^p(D)}+\sum_{j=1}^m ||Z_{j,n}||^p_{L^p(Q_t)} + t^p+t^{p-1}Y_n(t)\right).\]
   This is the exact analog of \eqref{eq: Ineq for Y(t)}.  The rest of the proof is now concluded by mimicking line-by-line the proof of Theorem \ref{thm: Theorem 2xr} and invoking Lemmas \ref{lemma: For every p there exists a T_p, C_p such that bound for small times} and \ref{lemma: lift L^p integral to L^infinity}.
\end{proof}

\bibliography{References}
\bibliographystyle{plain}

\end{document}